\documentclass{amsart}

\usepackage[utf8]{inputenc}
\usepackage[T2A]{fontenc}
\usepackage[russian,english]{babel}
\usepackage{amsmath}
\usepackage{amssymb}
\usepackage{amsthm}
\usepackage{amsfonts}
\usepackage{csquotes}
\usepackage{multirow}
\usepackage{datetime}

\newdateformat{monthyeardate}{\monthname[\THEMONTH], \THEYEAR}

\usepackage[english]{babel}

\usepackage[letterpaper,top=2.5cm,bottom=2.5cm,left=3cm,right=3cm,marginparwidth=1.75cm]{geometry}

\usepackage{amssymb}
\usepackage[utf8]{inputenc}
\usepackage[english]{babel}
\usepackage{listings}
\usepackage{color}
\usepackage{setspace}
\usepackage{hyperref}
\usepackage{acro}
\usepackage{amsmath}
\usepackage{mathrsfs}
\usepackage{amsthm}
\usepackage{graphicx}
\usepackage{geometry}
\usepackage{subcaption}
\usepackage{cancel}
\usepackage{tikz}
\usepackage{color}
\usepackage{comment}
\usepackage{enumerate}
\usepackage{amsfonts}

\usepackage{tikz-cd}
\usepackage{mathtools}
\usepackage{MnSymbol}
\usepackage{xfrac}
\usepackage{faktor}
\usepackage{xcolor}
\usepackage{comment}
\usepackage{braket}
\usepackage{bbm}
\usepackage{array}
\usepackage{float}

\usepackage{lipsum}
\usepackage{blindtext}
\usepackage{subfiles}
\usepackage[titletoc,title]{appendix}
\usepackage[most]{tcolorbox}
\usepackage[new]{old-arrows}
\usepackage{tikz}
\usetikzlibrary {decorations.markings,arrows, decorations.pathmorphing, decorations.pathreplacing, decorations.shapes}
\usepackage[backend=biber, style=alphabetic, sorting=nyt,maxbibnames=99]{biblatex}
\newcommand{\N}{\mathbb{N}}
\newcommand{\Z}{\mathbb{Z}}
\newcommand{\F}{\mathbb{F}}

\newcommand{\lp}{\left(}
\newcommand{\rp}{\right)}

\newcommand{\K}{\mathcal{K}}
\newcommand{\dd}{\mathrm{d}}
\newcommand{\CC}{\mathbb{C}}
\newcommand{\OO}{\mathcal{O}}
\DeclareMathOperator{\tr}{tr}
\DeclareMathOperator{\Spec}{Spec}
\DeclareMathOperator{\Tor}{Tor}

\DeclareMathOperator{\bideg}{bideg}
\DeclareMathOperator{\UKC}{U_{\K}(\mathbb{C})}
\DeclareMathOperator{\UKQ}{U_{\K}(\F_q)}

\newtheorem{theorem}{Theorem}

\newtheorem{proposition}[theorem]{Proposition}
\newtheorem{corollary}[theorem]{Corollary}

\theoremstyle{definition} 
\newtheorem{definition}[theorem]{Definition}

\newtheorem{remark}[theorem]{Remark}
\newtheorem*{notation}{Notation}

\graphicspath{{./}}

\newcommand\restr[2]{{\left.\kern-\nulldelimiterspace #1 \littletaller \right|_{#2} }}

\newcommand{\littletaller}{\mathchoice{\vphantom{\big|}}{}{}{}}
\newcommand{\acts}{\mathrel{\reflectbox{\rotatebox[origin=c]{-180}{$\circlearrowleft$}}}}

\usepackage{fancyhdr,textcase}
\title{Mixed Hodge structure in the cohomology of moment-angle complexes}

\author{Nikita Andrusov}
\thanks{Faculty of Mathematics, National Research Institute Higher School of Economics, 20 Myasnitskaya st., Moscow, 101000, Russia \\ \texttt{andrusov.n@gmail.com}}
\date{\monthyeardate\today} 
\begin{document}

\begin{abstract}
This work studies the relationship between the multigrading in the cohomology of moment-angle complexes, introduced by Buchstaber and Panov in their computation of these cohomology groups, and the multigrading in the cohomology of the corresponding complements of coordinate subspace arrangements, arising from the mixed Hodge structure. It is shown that these multigradings coincide. This result refines a theorem from Yu. Eliyashev's 2013 dissertation.

Furthermore, results on the mixed Hodge structure in the cohomology of complements of coordinate subspace arrangements are applied to the study of the same varieties over finite fields. In particular, various formulas are obtained for the number of $\mathbb{F}_q$-rational points and for the associated global $\zeta$-functions.
\end{abstract}

\maketitle

{\flushleft
}

\section{Introduction}\label{introduction}
Michèle Audin in her article \cite{audin91} introduced a quotient construction of a toric variety. This construction was generalized by David Cox in his work \cite{Cox:1993fz}. If a toric variety $V_{\Sigma}$ is built using a smooth fan $\Sigma \subset \mathbb{R}^d$ (i.e. the set of generator vectors of rows of each cone can be completed to be a $\mathbb{Z}$-basis of the one-parameter subgroup lattice of the torus), the intersection of the fan with a unit sphere provides a simplicial complex $\K = S^{d-1} \cap \Sigma$. $\K$ is a simplicial complex on $n$ vertices, where $n$ is the number of 1-dimensional cones (rays) in $\Sigma$. The Cox's construction of a toric variety represents $V_{\Sigma}$ as a quotient of an $n$-dimensional quasi-affine toric variety $U_{\K}$ by the action of a torus $\mathbb{T}^{n-d}$, where $d$ is the dimension of $V_{\Sigma}$ and
$$U_{\K} = \mathbb{A}^n \setminus \bigcup_{I\notin \K}\{(z_1, \ldots, z_n) \in \mathbb{A}^n\;|\;\forall i\in I \;z_i=0\}$$

When working over the field $\CC$, $\UKC$ is a non-compact complex variety and $\mathbb{T}^{n-d}(\mathbb{C}) = \lp\mathbb{C}^{\times}\rp^{n-d}\cong $\\ $\cong\lp S^1\rp^{n-d}\times \lp\mathbb{R}_{>0}^{\times}\rp^{n-d}$. In particular, one can define an exponential action of $\mathbb{R}^{n-d} \cong \lp\mathbb{R}_{>0}^{\times}\rp^{n-d} \subset \mathbb{T}^{n-d}(\mathbb{C})$ on $\UKC$ as the induced action of a subgroup. The intermediate quotient $\mathcal{Z}_{\K} = \UKC / \lp\mathbb{R}_{>0}^{\times}\rp^{n-d}$ is called a moment-angle complex. In general $\mathcal{Z}_K$ is not even a topological manifold, but only a CW-complex, however in special cases of a complete fans $\Sigma$ (when $\K$ is a triangulated sphere) $\mathcal{Z}_{\K}$ is indeed a manifold. Victor Buchstaber and Taras Panov first introduced moment-angle complexes in their works \cite{buchstaberpanov99} and \cite{buchstaberpanov00}. They showed that $\UKC$ has a deformation retraction to $\mathcal{Z}_{\K}$ and computed the cohomology groups $H^*(\mathcal{Z}_{\K}, \Z)$ with the ring structure. We recall several notations and definitions to introduce theorems of Buchstaber and Panov:

\begin{notation}
    Let $I \subseteq \{1, \ldots, n\}$ denote $$v^I = \prod_{i \in I}v_i \in \mathbb{C}[v_1, \ldots, v_n]$$
\end{notation}

\begin{definition}\label{stanley-risener}
    Let $\K$ be a simplicial complex on $n$ vertices. A Stanley---Reisener ring of $\K$ is the following:
    $$\mathbb{C}[\K] = \mathbb{C}[v_1, \ldots, v_n]/J,\;J = \left(v^I\;|\;I\subseteq \{1, \ldots, n\},\;I\notin \K\right)$$
\end{definition}

\begin{theorem}[Theorem 4.7.5 from \cite{buchstaber2014torictopology}]
    $\mathcal{Z}_{\K}$ is homotopy equivalent to $\UKC$.
\end{theorem}

\begin{theorem}[Theorem 4.5.7 from \cite{buchstaber2014torictopology}, originally from \cite{Baskakov2002CohomologyOK}]\label{theorem bigraded components BP}
Let $\mathcal{K}$ be a simplicial complex on $n$ vertices, $\mathcal{Z}_{\K}$ be the corresponding moment-angle complex. Then
$$H^k(\mathcal{Z}_{\K}, \mathbb{C}) = \bigoplus_{\substack{s,t\in\mathbb{N}\\2t-s=k}}H^{-s,2t}(\mathcal{Z}_{\K}),$$
$$H^{-s,2t}(\mathcal{Z}_{\K}) = \Tor^s_{\mathbb{C}[v_1, \ldots, v_n]}(\mathbb{C}[\K], \mathbb{C})[t] = \bigoplus_{I \in \N^n,\; \sum_i I_i = t} \Tor^s_{\mathbb{C}[v_1, \ldots, v_n]}(\mathbb{C}[\K], \mathbb{C})[I]$$
where $\Tor^s_{\mathbb{C}[v_1, \ldots, v_n]}(\mathbb{C}[\K], \mathbb{C})[I]$ is a component with multigrading $I$, the multigrading in Tor-functor value is induced from the multidegree on $\mathbb{C}[v_1, \ldots, v_n]$. And $\Tor^s_{\mathbb{C}[v_1, \ldots, v_n]}(\mathbb{C}[\K], \mathbb{C})[t]$ is the component of total degree $t$ (total degree is the sum of components in multidegree).
\end{theorem}

\begin{theorem}[Theorem 4.5.4 from \cite{buchstaber2014torictopology}]\label{theorem_buchstaber_panov}
    Let $\K$ be a simplicial complex. Cohomology of a moment-angle complex $\mathcal{Z}_{\K}$ can be computed as cohomology of the following $dg$-algebra:
    $$\mathcal{A}(\K) = \Lambda[u_1,\ldots, u_n]\otimes\Z[\K]/(u_iv_i=v_i^2=0\;\,\mid\;i=1\ldots n),$$
    $$\dd u_i=v_i,\dd v_i=0,\;\bideg(u_i)=(-1,2),\bideg(v_i)=(0,2)$$
\end{theorem}

The explanation of the appearance of this decomposition of cohomology into summands is provided in Yuri Eliyashev's articles  \cite{eliyashev13}, \cite{eliyashev16} and his thesis \cite{eliyashev13diss}. Since $\mathcal{Z}_{\K}$ is homotopy equivalent to $\UKC$, their cohomology is equal. Eliyashev showed that the Buchstaber---Panov bigraded components of the cohomology might be expressed in terms of mixed Hodge structure in the cohomology of $\UKC$:

\begin{theorem}[Eliyashev, 2016]
Let $\K$ be a simplicial complex, $\UKC$ be the corresponding complement to the coordinate subspace arrangement, $\mathcal{Z}_{\K}$ be the corresponding moment-angle complex. Then

$$H^{-s, 2t}(\mathcal{Z}_{\K}, \mathbb{C}) = Gr^W_{2t}H^{2t-s}(\UKC, \mathbb{C})$$

where $Gr^W_{2t}H^{2t-s}(\UKC, \mathbb{C})$ are the associated graded pieces of the weight filtration (introduced by Pierre Deligne in \cite{deligne1970} and \cite{deligne1971}) on cohomology of $\UKC$.
\end{theorem}

In the present article continuing Eliyashev's work cohomology of complements of coordinate subspace arrangements is studied. Section \ref{structural sheaf coh} is an auxiliary part where we compute the cohomology of the structure sheaf on $\UKC$. In section \ref{mixed hodge structure}, using the results of section \ref{structural sheaf coh}, an alternative proof of Eliyashev's theorem is given. This proof is less technical and uses more of theory about toric varieties. Theorem \ref{main theorem} is a refinement of Eliyashev's theorem expressing components $H^{-s, 2t}(\mathcal{Z}_{\K}, \mathbb{C})[I] = \Tor^s_{\mathbb{C}[v_1, \ldots, v_n]}(\mathbb{C}[\K], \mathbb{C})[I]$ for $I \in \N^n$ in terms of cohomology of sheaves on $\UKC$. Then the results about mixed Hodge structure are applied to studying the number of $\F_q$-rational points of $U_{\K}$. In section \ref{finite fields} we continue the study of $U_{\K}$ as a scheme over $\Z$: combinatorial formulas for the number of $\F_q$-rational points and for the global $\zeta$-function of $U_{\K}$ are provided. The number of $\F_q$-rational points is expressed in terms of Hilbert---Poincare series for $\mathbb{C}[\K]$ in theorem \ref{prop hilbert poincare} and in terms of dimensions of Buchstaber---Panov components of $H^*(\mathcal{Z}_{\K}, \mathbb{C})$ in theorem \ref{theorem points count cohomology}.
Finally, Section \ref{examples} contains three important examples, which help illustrate the constructions discussed in this article.

\subsection*{Definitions and notations} 
We write $\mathbb{A}^n$ for the $n$-dimensional affine space, $\mathbb{P}^n$ for the $n$-dimensional projective space, $\mathbb{T}^n$ for the $n$-dimensional algebraic torus, in particular $\mathbb{T}^n(\mathbb{C}) = \left(\mathbb{C}^{\times}\right)^n$.

\begin{definition}
    Let $A \supseteq B$ be topological spaces or algebraic varieties, $I \subseteq \{1, 2, \ldots, n\}$. Denote $$(A, B)^I = \prod_{i=1}^n C_i,\text{ where } C_i = \begin{cases}
        A,\text{ if } i \in I\\
        B,\text{ otherwise }
    \end{cases}$$
    Note that $(A, B)^I \subseteq A^n$.
    For $\K$ -- a simplicial complex we denote 
    $$(A, B)^{\K} = \bigcup_{I \in \K}(A, B)^I \subseteq A^n$$
\end{definition}

\begin{remark}[Proposition 4.7.3 from \cite{buchstaber2014torictopology}]\label{remark_union}
$U_{\K}$ can be expressed as follows:
$$U_{\K} = \bigcup_{\sigma \in \K}(\mathbb{A}^1, \mathbb{T}^1)^I = (\mathbb{A}^1, \mathbb{A}^1 \setminus \{0\})^{\K}$$
since $\mathbb{A}^1$ and $\mathbb{A}^1\setminus \{0\}$ are affine varieties, the union above is the affine covering of $U_{\K}$.

Also
$$U_{\K} = \mathbb{A}^n \setminus \bigsqcup_{J \notin \K}\left\{(x_1, \ldots, x_n) \in \mathbb{A}^n\;\mid\;\forall\;j\in J\;x_j=0,\;\forall \;j\notin J\;x_j\neq0\right\}$$
\end{remark}

\begin{notation}
For a finite set $S$, we denote by $\#S$ the number of its elements.
\end{notation}

\begin{definition}[$f$-vector of a simplicial complex]
    Let $d$ be the highest dimension of a simplex in $\K$. Let $k \in \{-1, \ldots, d\}$. We denote by $\K^{(k)}$ the set of all simplices of dimension $k$ in $\K$. Then $f(\K) = \lp f_k = \#\K^{(k)}\rp_{k=-1}^d$ is called the $f$-vector of $\K$.
\end{definition}

\begin{comment}
\begin{definition}
Let $D^2$ be a 2-dimensional unit disk, $S^1$ be the disk's boundary circle. Then $\mathcal{Z}_{\K} = (D^2, S^1)^{\K}$ is called a moment-angle complex.
\end{definition}
\end{comment}

\begin{notation}
In the context of toric geometry we write $M\cong \Z^n$ for the character lattice, $N\cong \Z^n$ for its dual. We fix a basis $\{e_i\}_{i=1}^n$ of $N$ and we write $m = (m_1, \ldots, m_n) \in M$ meaning $m_i$ are the coordinates of $m$ in the dual basis to $\{e_i\}$.
\end{notation}

\begin{comment}
\begin{notation}
For two cones $\sigma, \sigma' \subset \mathbb{R}^n$ we write $\sigma' \prec \sigma$ if $\sigma'$ is an edge of $\sigma$.
\end{notation}
\end{comment}

\begin{notation}
We denote $F^sH^*(X, \mathbb{C})$ the $s$-th element of Hodge filtration on cohomology of complex algebraic variety $X$ and $W_wH^*(X, \mathbb{C})$ the  $w$-th element of weight filtration on cohomology of algebraic variety $X$. By $Gr^W_wH^*(X, \mathbb{C})$ we denote the associated graded pieces of the weight filtration: $Gr^W_wH^*(X, \mathbb{C})$ = $W_wH^*(X, \mathbb{C})/W_{w-1}H^*(X, \mathbb{C})$. Similarly by $Gr_F^sH^*(X, \mathbb{C})$ we denote the associated graded pieces of Hodge filtration.
\end{notation}

\subsection*{Acknowledgments}
The author is deeply thankful to his scientific supervisor Victor V. Batyrev. The author is also thankful to his teachers: George B. Shabat's courses on algebraic geometry and Hodge theory had an important impact on author's knowledge in these topics. Alexander A. Gaifullin has inspired the author to study algebraic topology. Taras E. Panov's lecture course on toric topology helped a lot in understanding the topic of moment-angle complexes. 

Vladislav Mikhailov explained to the author the construction of combinatorial Alexander duality, although it is not presented in the final version of the paper, it was a key part of the first versions of the proof of Theorem \ref{main theorem}.

This research was done as a bachelor thesis at the School of Applied Mathematics and Computer Science at the Moscow Institute of Physics and Technology.

\section{Cohomology of a structure sheaf on $\UKC$}\label{structural sheaf coh}

Victor Buchstaber and Taras Panov proved that each bigraded component of cohomology of a moment-angle complex decomposes as a direct sum of multigraded components, each of which can be expressed in terms of the cohomology of subcomplexes of $\K$:

\begin{theorem}[Theorem 4.5.8 from \cite{buchstaber2014torictopology}]\label{cohomology from complex thm}

$$H^{-s, 2t}\lp \UKC, \mathbb{C} \rp = \bigoplus_{\substack{J\subseteq\{1, 2, \ldots, n\}\\ \# J=t}} H^{-s, 2t}\lp \UKC, \mathbb{C} \rp[J]$$

$$H^{-s, 2t}\lp \UKC, \mathbb{C} \rp[J] = \Tor_{\mathbb{C}[z_1, \ldots, z_n]}(\mathbb{C}[\K], \mathbb{C})[J] = \widetilde{H}^{t-s-1}\lp \K|_{J}\rp$$
where $\Tor_{\mathbb{C}[z_1, \ldots, z_n]}(\mathbb{C}[\K], \mathbb{C})[J]$ is the component of the value of the Tor functor in multidegree $J$ (where subsets $J\subseteq \{1, \ldots, n\}$ are identified with vectors $\{0,1\}^n \subset \N^n$, and the multigrading is induced by the multidegree on $\mathbb{C}[z_1, \ldots, z_n]$-modules).\
$\widetilde{H}^{t-s-1}\lp \K|_{J}\rp$ denotes the reduced cohomology of the simplicial complex obtained from $\K$ by restricting to the vertices in the set $J$.
\end{theorem}

In this section, we show that each of these multigraded components can be expressed in terms of the cohomology of the structure sheaf on $\UKC$. In Section \ref{mixed hodge structure}, using the results of this section, we will relate this multigrading to the multidegree of differential forms arising in de Rham cohomology, which will prove the main theorem of this work (Theorem \ref{main theorem}).

It is known (see Example 5.1.9 from \cite{buchstaber2014torictopology}) that $U_{\K}$ is a toric variety. Below we recall the structure of the corresponding fan:

\begin{notation}
Here and below, for $I \subseteq \{1, \ldots, n\}$, we denote $\{1, \ldots, n\} \setminus I$ by $\overline{I}$.
\end{notation}

Define $\psi: \{1,\ldots, n\}\rightarrow N \otimes \mathbb{R},\;\psi(i) = e_i$.

\begin{proposition}
$\UKC$ is a toric variety of dimension $n$.  $\UKC$ can be defined by the fan $\Sigma = \left\{\mathbb{R}_{\geqslant 0}\left\langle\psi(\sigma)\right\rangle\;|\;\sigma \in \K\right\}$, where $\psi(\sigma)$ is the image of the set $\sigma \subseteq \{1, \ldots, n\}$.
\end{proposition}
\begin{proof}

Denote the toric variety corresponding to the fan $\Sigma$ by $V_{\Sigma}$.
The affine chart corresponding to a cone $\sigma \subseteq N \otimes \mathbb{R}$ is $\Spec\lp\mathbb{C}[\check{\sigma} \cap M]\rp$. Note that for $\sigma = \mathbb{R}_{\geqslant 0}\langle\psi(I)\rangle,\;I\subseteq\{1, \ldots, n\}$, since $\psi(I) = \{e_i\;|\;i\in I\}$, we have $\check{\sigma} = \lp\mathbb{R}, \mathbb{R}_{\geqslant 0}\rp^{\overline{I}}$. Thus, for $i \in I$, this cone contains only monomials with nonnegative $i$-th coordinate, while for $i\notin I$ the cone contains monomials with arbitrary $i$-th coordinates. Hence
$$\mathbb{C}[\check{\sigma}\cap M] = \mathbb{C}[u_1, \ldots, u_n],\;u_i = \begin{cases}
v_i,\text{ if }i \in I\\
v_i^{\pm 1},\text{ if }i \notin I
\end{cases}$$
and therefore
$$\Spec\lp\mathbb{C}[\check{\sigma}\cap M]\rp = \prod_{i=1}^n C_i = \lp\mathbb{C}, \mathbb{C}^{\times}\rp^{I},\;C_i = \begin{cases}
\mathbb{C} = \Spec(\mathbb{C}[v_i]),\text{ if } i \in I\\
\mathbb{C}^{\times} = \Spec(\mathbb{C}[v_i^{\pm 1}]),\text{ if }i \notin I
\end{cases}$$
Therefore $V_{\Sigma}$ is the union of such affine charts, i.e.
$$\UKC = \lp\mathbb{C}, \mathbb{C}^{\times}\rp^{\K} = \bigcup_{I \in \K}\lp\mathbb{C}, \mathbb{C}^{\times}\rp^I = V_{\Sigma}$$

\end{proof}

\begin{remark}\label{remark fan criteria}
One can say that $U_{\K}$ are precisely those $n$-dimensional toric varieties whose fan is simplicial and has exactly $n$ one-dimensional cones, with the generators of these cones forming a basis of the lattice $N$. Conversely, all such toric varieties are of the form $U_{\K}$ for some $\K$.
\end{remark}
\begin{proof}
The first direction was proved in the previous proposition, while for the converse it is enough to choose a basis of $N$ whose basis vectors are the generators of the one-dimensional cones of the fan.
\end{proof}

On a toric variety structure sheaf, sheafs of differential forms and consequently all the cohomology of those sheaves are naturally $M$-graded as shown for example in the review by Danilov \cite{Danilov1978-bq}.

\begin{notation}
For $m \in M$, denote by $H^*(\UKC, \mathcal{O}_{\UKC})[m]$ the component of degree $m$ in the cohomology of the structure sheaf.
\end{notation}

\begin{definition}
Let $m = (m_1, \ldots, m_n) \in M$. Define $I(m) = \{i\;\mid\;m_i < 0\}$.
\end{definition}

\begin{theorem}\label{main proposition structural sheaf}
\(\widetilde{H}^{i-1}\lp\K|_{I(m)}, \mathbb{C}\rp\cong H^i(\UKC, \mathcal{O}_{\UKC})[m]\)
\end{theorem}
\begin{proof}
It is known from the book \cite{cox2011toric} (Chapter 9, \S 1) that for a toric variety $V_{\Sigma}$, $H^i(V_{\Sigma}, \mathcal{O}_{V_{\Sigma}})[m]\cong \widetilde{H}^{i-1}(X_m, \mathbb{C})$, where 
$$X_m = \bigcup_{\sigma \in \Sigma}Conv\lp u\;|\;u \text{ is a generator of } \sigma,\;\langle m, u\rangle < 0\rp$$
In our case, all such $u$ are basis vectors of the lattice $N$. Accordingly, the condition $\langle m, u\rangle < 0$ means that $u$ is the $i$-th basis vector for $i\in I(m)$. Thus, for every $\sigma \in \Sigma, \sigma = \mathbb{R}_{\geqslant 0}\langle\psi(I)\rangle$, where $I \in \K$, we have
$Conv(u\;|\;u\text{ is a generator of } \sigma,\;\langle m,u\rangle < 0) = Conv\lp\psi\lp I\cap I(m)\rp \rp$. Therefore $X_m\cong \K|_{I(m)}$.
\end{proof}

Applying Theorem \ref{main proposition structural sheaf}
to Theorem \ref{cohomology from complex thm}, we obtain the main theorem of this section:

\begin{theorem}\label{main theorem structural sheaf}
Let $\K$ be a simplicial complex, and let $\UKC$ be the corresponding complement to coordinate subspace arrangement. Let $I\subseteq \{1, \ldots, n\}$ and $m \in M,\;I(m)=I$. Then

$$H^{-s, 2t}\lp\UKC, \mathbb{C}\rp[I]\cong H^{t-s}\lp \UKC, \mathcal{O}_{\UKC}\rp[m]$$

\end{theorem}

We will also need some technical observations on the \v{C}ech complex for $\OO_{\UKC}$:

Recall that $\Spec(\mathbb{C}[\check{\sigma}_1\cap M])\cap Spec\lp\mathbb{C}[\check{\sigma}_2\cap M]\rp = \Spec(\mathbb{C}[(\sigma_1\cap\sigma_2)^{\lor}\cap M])$,
where $\check{\sigma}$ is the dual cone to $\sigma$.
Therefore for a sheaf $\mathcal{S}$ its \v{C}ech complex is the following:

$$0\rightarrow \bigoplus_{\sigma \in \Sigma} \mathcal{S}\lp \Spec(\mathbb{C}[\check{\sigma}\cap M])\rp\xrightarrow[]{\mathrm{d}} \bigoplus_{\sigma_1, \sigma_2 \in \Sigma, \sigma_1\neq\sigma_2} \mathcal{S}\lp \Spec(\mathbb{C}[\lp\sigma_1\cap\sigma_2\rp^{\lor}\cap M])\rp\xrightarrow[]{\mathrm{d}} \ldots$$

For $I \in \K$ we write $V_{I}$ for the chart $\Spec(\mathbb{C}[\mathbb{R}_{\geqslant 0}\langle\psi(I)\rangle^{\lor}\cap M]) = (\CC, \CC^{\times})^I$. Then the \v{C}ech complex can be written as:

$$0\rightarrow\bigoplus_{I \in \K}\mathcal{S}\lp V_{I}\rp\xrightarrow[]{\mathrm{d}}\bigoplus_{I_1, I_2 \in \K, I_1\neq I_2}\mathcal{S}\lp V_{I_1\cap I_2}\rp\xrightarrow[]{\mathrm{d}}\ldots\rightarrow 0$$

Now consider an invertible sheaf $\mathcal{S} = \mathcal{O}_{\UKC}[m]$ -- a component of the structure sheaf in grading $m$. It is an invertible sheaf of 0-form $\chi^m$ -- a character corresponding to $m \in M$. Then its \v{C}ech complex is:
$$0\rightarrow\bigoplus_{I \in \K}\mathcal{O}_{\UKC}[m]\lp V_{I}\rp\xrightarrow[]{\mathrm{d}}\bigoplus_{I_1, I_2 \in \K, I_1\neq I_2}\mathcal{O}_{\UKC}[m]\lp V_{I_1\cap I_2}\rp\xrightarrow[]{\mathrm{d}}\ldots\rightarrow 0$$

\begin{remark}\label{section of cech complex rem}
Since for $I \in \K$ $V_{I} = \lp\mathbb{C}, \mathbb{C}^{\times}\rp^I$, $\mathcal{O}_{\UKC}(V_{I}) = \mathbb{C}[x_1, \ldots,x_n,x_{i_1}^{-1}, \ldots, x_{i_k}^{-1}]$, where $\{i_1, \ldots, i_k\} = \overline{I}$, therefore
$$\mathcal{O}_{\UKC}[m](V_{I}) = \begin{cases}
\mathbb{C},\text{ if } I \cap I(m) = \emptyset\\
0,\text{ otherwise}    
\end{cases}$$
\end{remark}

We again observe the following property of the cohomology of structure sheaf in different graded components:
\begin{proposition}\label{prop_chech}
    If $m_1, m_2 \in M$ such that $I(m_1) = I(m_2)$. Then \\$H^*(\UKC, \mathcal{O}_{\UKC})[m_1] = H^*(\UKC, \mathcal{O}_{\UKC})[m_2]$
\end{proposition}

and the following equivalence is motivated now:
\begin{notation}
    We denote by $[M]$ the coset $M/\sim$, where for $m,m'\in M$ $m\sim m'\Leftrightarrow I(m)=I(m')$. The \v{C}ech complex of $\OO_{\UKC}[m]$ depends only on $[m] \in [M]$.
\end{notation}

We denote by $\mathcal{C}_{[m]}^{\bullet}$ the \v{C}ech complex for $\OO_{\UKC}[m]$.

\section{Mixed Hodge structure}\label{mixed hodge structure}
\subsection{Plan of the section}
We can now apply the results of the previous section to compute mixed Hodge structure in the cohomology of $\UKC$.

For complete varieties the holomorphic \v{C}ech---de Rham double complex is used to compute the Hodge filtration on cohomology:

\begin{definition}[Holomorphic \v{C}ech---de Rham double complex]\label{chech de rham}
    A double complex $\mathcal{E}^{\bullet,\bullet}$, where
    $\mathcal{E}^{s,\bullet}$ is \v{C}ech complex for sheaf $\Omega^s_{X}$ and $\mathcal{E}^{\bullet, t}$ is a sum of de Rham complexes for each chart.
\end{definition}

It is known that Hodge filtration could be calculated using holomorphic \v{C}ech---de Rham double complex (see 2.4.1 from \cite{cattani2014hodge} for reference):

\begin{theorem}\label{stupid filtration}
For $X$ --- a complete variety hypercohomology of $\mathcal{E}^{\bullet, \bullet}$ are the cohomology of $X$: 
    $$\mathbb{H}^i(\mathcal{E}^{\bullet, \bullet}) = H^i\lp X, \mathbb{C}\rp$$
    The Hodge filtration could be expressed as follows: 
    $$F_kH^i\lp X, \mathbb{C}\rp = Im\lp\mathbb{H}^i\lp\mathcal{E}^{\geqslant k, \bullet}\rp \rightarrow \mathbb{H}^i(\mathcal{E}^{\bullet, \bullet})\rp$$
\end{theorem}

In the case of non-complete $X$ (such as all of $U_{\K}$ are) Pierre Deligne in his work \cite{deligne1971} introduced another way to compute the Hodge filtration on cohomology:

\begin{definition}[Proper compactification]
Let $X$ be a non-complete smooth variety. Then $\overline{X} \hookleftarrow X$ is called a proper compactification of $X$ if $D = \overline{X}\setminus X$ is a simple normal crossing divisor and $\overline{X}$ is complete and smooth.
\end{definition}

\begin{definition}
Let $I = \{i_1, \ldots, i_s\} \subseteq \{1, \ldots, n\}$. We denote $$\lp\frac{\mathrm{d}x}{x}\rp_I = \frac{\mathrm{d}x_{i_1}}{x_{i_1}}\land\frac{\mathrm{d}x_{i_2}}{x_{i_2}}\land\ldots\land\frac{\mathrm{d}x_{i_s}}{x_{i_s}}$$
Also denote 
$$\lp\mathrm{d}x\rp_{I} = \mathrm{d}x_{i_1}\land \mathrm{d}x_{i_2}\land \ldots\land \mathrm{d}x_{i_s}$$
\end{definition}

\begin{definition}[Form with at worst logarithmic singularities in $D$]
Let $\overline{X}$ be a proper compactification for $X$. A form on $\overline{X}\;\omega$ is called a form with at worst logarithmic singularities in $D$ if in every affine chart $U$ (since $D$ is an SNC divisor we consider $D\cap U = \{y_1,\cdots, y_k=0\}$, $y_1, \ldots, y_n$ are the coordinates in chart $U$) is decomposable into sum of forms:
$$\omega = \sum_i \omega_i,\;\omega_i = f_i \cdot \lp\frac{\dd x}{x}\rp_I\land \lp\dd x\rp_J,\;I\subseteq \{1, \ldots, k\}$$
\end{definition}

\begin{notation}
Let $\overline{X}$ be a proper compactification for $X,\;D=\overline{X}\setminus X$. We denote by $\Omega^s_{\overline{X}}(\log D)$ the sheaf of $s$-forms with at worst logarithmic singularities in $D$.
\end{notation}

\begin{definition}[Logarithmic \v{C}ech---de Rham complex]
Let $\overline{X}$ be a proper compactification for $X,\;D=\overline{X}\setminus X$.
 A double complex $\mathcal{E}_{\log}^{\bullet,\bullet}$, where
    $\mathcal{E}_{\log}^{s,\bullet}$ is \v{C}ech complex for sheaf $\Omega^s_{\overline{X}}(\log D)$ and $\mathcal{E}_{\log}^{\bullet, t}$ is a sum of de Rham complexes for each chart, is called logarithmic \v{C}ech---de Rham complex.
\end{definition}

\begin{theorem}[From \cite{deligne1971}]
For $X$ --- a smooth variety and $\overline{X}$ --- a proper compactification for $X,\;D=\overline{X}\setminus X$. Hypercohomology of $\mathcal{E}_{\log}^{\bullet, \bullet}$ are the cohomology of $X$: 
    $$\mathbb{H}^i(\mathcal{E}_{\log}^{\bullet, \bullet}) = H^i\lp X, \mathbb{C}\rp$$
    The Hodge filtration could be expressed as follows: 
    $$F_kH^i\lp X, \mathbb{C}\rp = Im\lp\mathbb{H}^i\lp\mathcal{E}_{\log}^{\geqslant k, \bullet}\rp \rightarrow \mathbb{H}^i(\mathcal{E}_{\log}^{\bullet, \bullet})\rp$$
\end{theorem}

The filtration from theorem \ref{stupid filtration} could be constructed in smooth non-complete case, but in general it induces another filtration called ''stupid filtration'' on cohomology. However in the specific case of $U_{\K}$ stupid filtration coincides with Hodge filtration, which is proven below\footnote{I believe this fact should be a corollary of some more general fact for mixed Hodge structures on toric varieties, but I failed to find any reference in the literature stating the coincidence of stupid filtration and Hodge filtration, so I provide a proof of this coincidence in the case I need.}. 

The plan of this section is the following: At first we prove that stupid filtration induces Hodge filtration. The computation of stupid filtration is reduced to the computation of the cohomology of different $M$-graded components of structure sheaf performed in section \ref{structural sheaf coh}. Then using results from \cite{Weber2004} on mixed Hodge structure in cohomology of toric varieties the weight filtration is computed giving an alternative proof for Eliyashev's theorem. Then a refinement of Eliyashev's theorem is presented (Theorem \ref{main theorem}) and then it is applied to counting $\F_q$-rational points on $U_{\K}$ using Grothendieck---Lefschetz trace formula.

\subsection{Coincidence of filtrations}

\begin{remark}
	Danilov in \S 7 from \cite{Danilov1978-bq} showed that for a complete toric $X$ the holomorphic \v{C}ech---de Rham complex is $M$-graded and hence hypercohomology is also $M$-graded. Since in each affine chart spaces of sections for $m \in M$ of any sheaf of forms for gradings $m$ and $lm$ for any $l \in \mathbb{N}$ are isomorphic,
	$$\mathbb{H}^i(\mathcal{E}^{\geqslant k, \bullet})[m]\cong\mathbb{H}^i(\mathcal{E}^{\geqslant k, \bullet})[lm]$$
	Since all cohomology groups of $\UKC$ are finite-dimensional, all the nontrivial image of $\mathbb{H}^i(\mathcal{E}^{\geqslant k, \bullet})$ in $H^i(\UKC)$ must be concentrated in grading $0$.
	
	This idea actually needs neither completeness of $X$, nor usage of sheafs of forms. So the same is true for $\mathcal{E}^{\bullet, \bullet}$ for non-complete toric $X$ and for $\mathcal{E}_{\log}^{\geqslant k, \bullet}$ --- the logarithmic \v{C}ech---de Rham complex on a proper compactification of $X$.
\end{remark}

\begin{remark}
    Thus if a proper compactification $\overline{\UKC}$ of $\UKC$ is provided, the problem of computing stupid filtration and Hodge filtration on cohomology of $\UKC$ is reduced to studying holomorphic (for $\UKC$) and logarithmic (for $\overline{\UKC}, D=\overline{\UKC}\setminus\UKC$) \v{C}ech---de Rham complex in grading 0, where all de Rham differentials vanish. That is, we are dealing with a collection of separate \v{C}ech complexes: one for each $\Omega^s_{\UKC}$ ($\Omega^s_{\overline{\UKC}}(\log D)$ respectively).
\end{remark}

Firstly we shall say that for $\K = \Delta^{n-1}$ $\UKC = \mathbb{A}^n$ and all the results of this section are trivial since $$H^*(\mathcal{Z}_{\K}) = H^0(\mathcal{Z}_{\K}) = H^{0,0}(\mathcal{Z}_{\K}) = H^{0,0}(\mathcal{Z}_{\K})[\emptyset] = H^*(\mathbb{A}^n) = Gr^W_0 H^0(\mathbb{A}^n) = Gr^W_0 H^0(\mathbb{A}^n)[\emptyset]$$
so from now on we consider $\K\neq \Delta^{n-1}$.

Let us show the construction of $\overline{\UKC}$ --- a proper compactification of $\UKC$. We construct it as a toric variety too with a fan $\overline{\Sigma} \hookleftarrow \Sigma$ for $\Sigma$ --- a fan for $\UKC$.

\begin{remark}
For $\overline{\UKC}$ to be a proper compactification of $\UKC$ $\overline{\Sigma}$ should be a smooth complete fan such that each of the new cones (the ones from $\overline{\Sigma}\setminus \Sigma$) contains at least one new ray. 

This is true because in this case $D = \overline{\UKC}\setminus \UKC$ is the union of open subsets of $D_{\rho}$ --- torus invariant divisors for $\rho$ --- new rays. So the completion is indeed a divisor and since $D_{\rho}$ are torus invariant divisors on a smooth variety, $D$ is simple with normal crossings.
\end{remark}

Now the construction of $\overline{\Sigma}$ is rather straightforward:
we add a ray generated by $\mathbf{-1}=(-1, -1, \ldots, -1)$ in the standard basis of $N$ and for each $I \notin \K$ we add a ray generated by $u_I = \sum_{i \in I}e_i$.

And the added cones are the following:
Consider an order on $2^{\{1, \ldots, n\}}$ by inclusion. Then for each increasing chain of sets $I, I_1, \ldots, I_k \subseteq \{1, \ldots, n\}$ with $k>0$, where $I \in \K,\;I_1\notin \K$ we add a cone $$\sigma = \mathbb{R}_{\geqslant 0}\left\langle \{e_i\;|\;i\in I\}\cup \left\{\sum_{i \in I_j}e_i\;|\;j=1,\ldots, k\right\}\right\rangle$$
and if $I_k\neq \{1, \ldots, n\}$ we also add a cone
$$\sigma = \mathbb{R}_{\geqslant 0}\left\langle \{e_i\;|\;i\in I\}\cup \left\{\sum_{i \in I_j}e_i\;|\;j=1,\ldots, k\right\}\cup \{-\mathbf{1}\}\right\rangle$$

\begin{proposition}
Let $\overline{\UKC}$ be a toric variety with described above fan $\overline{\Sigma}$. Then $\overline{\UKC}$ is a proper compactification of $\UKC$.
\end{proposition}
\begin{proof}
We show below that $\overline{\Sigma}$ is indeed a complete fan.
It is obvious from the construction that $\overline{\Sigma}$ is smooth. Since in each chain $I, I_1, \ldots, I_k \subseteq \{1, \ldots, n\}$ $k$ is considered to be strictly positive, each added cone contains a ray corresponding to $I_k$ which is new. 

We denote by $\Sigma_{\K}$ the fan for $U_{\K}$.
Now to finish the proof let us show by induction on $\K$ that $\overline{\Sigma_{\K}}$ is indeed a complete fan. The induction is by the number of excluded from $\K$ simplices. The induction base is $\K = \partial\Delta^{n-1}$ and the only excluded simplex is $\{1, \ldots, n\}$. Then $\overline{\Sigma}$ is a fan for $Bl_0\mathbb{P}^n$ -- blowing up a projective space at 0, which is indeed a complete fan.

Let $\Sigma_{\K}$ be a fan for which the induction hypothesis is proven and $\K' = \K\setminus \{\sigma\}$, where $\sigma$ is inclusion-maximal simplex in $\K$. Consider any cone from $\overline{\Sigma_{\K}}$ corresponding to the chain $\sigma, I_1, \ldots, I_k$. What happens to it when replacing $\overline{\Sigma_{\K}}$ with $\overline{\Sigma_{\K'}}$ is basically just subdivision of the face corresponding to $\sigma$ by its median ray.

Now it is easy to see that this cone becomes a union of properly intersecting cones of the form $\sigma\setminus \{i\}, \sigma, I_1, \ldots, I_k$ for $i\in\sigma$ with the boundary being a union of properly intersecting cones of the form $\sigma\setminus \{i\}, \sigma$ and $\sigma\setminus \{i\}, I_1, \ldots, I_k$.
Notice that the same is true by the same argument for the ''second-type'' cones --- the ones with the vector $\mathbf{-1}$.
Thus the induction step is proven.
\end{proof}

To prove the theorem on coincidence of stupid filtration and Hodge filtration we need to recall a theorem from combinatorial topology:

\begin{definition}
In a simplicial complex $\widetilde{\K}$ vertices $u, v$ are called equivalent if $\forall I \in \widetilde{\K}, u\in I \Rightarrow I\cup \{v\} \in \widetilde{\K}$ and $\lp I\setminus \{u\}\rp \cup \{v\} \in \widetilde{\K}$ and $\forall I \in \widetilde{\K}, v\in I \Rightarrow I\cup \{u\} \in \widetilde{\K}$ and $\lp I\setminus \{v\}\rp \cup \{u\} \in \widetilde{\K}$
\end{definition}

\begin{theorem}[See for reference Proposition 5.1.8 from \cite{barmak2011algebraic}]\label{equivalent points homotopy}
If $u, v$ are equivalent vertices in $\widetilde{\K}$ then the quotient simplicial complex $\widetilde{\K}\{/u=v\}$ is homotopy equivalent to $\widetilde{\K}$.
\end{theorem}

\begin{remark}\label{rem cech equivalence}
The same notion of equivalence can be applied to \v{C}ech complexes and charts. Two charts $U_k, U_j$ of the affine covering $\{U_i\}_i$ are called equivalent for a sheaf $\mathcal{S}$ if the sections are the same for all $i_1, \ldots, i_s$:
$$\mathcal{S}\lp U_k\cap U_{i_1}, \cap \ldots \cap U_{i_s}\rp = \mathcal{S}\lp U_j\cap U_{i_1}, \cap \ldots \cap U_{i_s}\rp$$
Then one of the equivalent charts can be excluded from the covering and by Theorem \ref{equivalent points homotopy} the \v{C}ech cohomology for the new covering remain the same as for the previous.
\end{remark}

\begin{notation}
The cone in $\overline{\Sigma_{\K}}$ corresponding to an increasing chain of subsets of $\{1, \ldots, m\}$ $I, I_1, \ldots, I_k$ we denote by $(I, I_1, \ldots, I_k)$. Recall that cones $(I)$ (i.e. $k=0$) are exactly those presented in $\Sigma_{\K}$.
\end{notation}

\begin{proposition}\label{prop holomorphic sections}
Denote by $\check{\mathcal{C}}_s^{\bullet}$ the \v{C}ech complex for $\Omega^s_{\UKC}$ in grading 0. Then this complex splits into a direct sum:
$$\check{\mathcal{C}}_{s}^{\bullet} = \bigoplus_{J \subset \{1, \ldots, m\}}\check{\mathcal{C}}^{\bullet}\lp\lp\frac{\dd x}{x}\rp_J\rp$$
where $\check{\mathcal{C}}^{\bullet}\lp\lp\frac{\dd x}{x}\rp_J\rp$ is the invertible sheaf of the form $\lp\frac{\dd x}{x}\rp_J$ and its section space in chart corresponding to a cone $I$ is either trivial (i.e. $=0$) or 1-dimensional. And it is 1-dimensional if and only if $J\cap I = \emptyset$.
\end{proposition}
\begin{proof}
Each affine chart of $\UKC$ corresponds to a simplex $I \in \K$ and has the form $(\mathbb{C}, \mathbb{C}^{\times})^I$. Therefore, differential forms of degree $s$ on this chart have the form:
    $$\sum_{J\subseteq \{1, \ldots, n\},\;\#I=s}P(x_1, \ldots, x_n)\lp\mathrm{d}x\rp_J$$
    where $P(x_1, \ldots, x_n)$ is a polynomial in all $x_i, i\in I$ and a Laurent polynomial in all $x_i, i \notin I$. Since $\lp\mathrm{d}x\rp_i$ has grading $(0, \ldots, 0, 1, 0, \ldots, 0)$ --- a one in the $i$-th position, only a monomial of the form $\lp\frac{\mathrm{d}x}{x}\rp_J$ can have grading 0, and such a monomial exists on this chart only when $J \cap I = \emptyset$, which proves the second statement of the proposition. 
    
    Denote by $\check{\mathcal{C}}^{\bullet}\left(\lp\frac{\mathrm{d}x}{x}\rp_J\right)$ the \v{C}ech complex for the invertible sheaf of the differential form $\lp\frac{\mathrm{d}x}{x}\rp_J$.
    
    Then each term of the \v{C}ech complex splits into a direct sum:
    $$\check{\mathcal{C}}_s^i = \bigoplus_{J\subseteq \{1, \ldots, n\},\; \#J=s}\check{\mathcal{C}}^{i}\left(\lp\frac{\mathrm{d}x}{x}\rp_J\right)$$
    Note that the \v{C}ech differential is compatible with this direct sum decomposition, so the entire \v{C}each complex splits into a direct sum of complexes:
    $$\check{\mathcal{C}}_s^{\bullet} = \bigoplus_{J\subseteq \{1, \ldots, n\},\; \#J=s}\check{\mathcal{C}}^{\bullet}\left(\lp\frac{\mathrm{d}x}{x}\rp_J\right)$$
\end{proof}

\begin{proposition}\label{prop log sections}
Denote by $\check{\mathcal{C}}_{s,\log}^{\bullet}$ the \v{C}ech complex for $\Omega_{\overline{\UKC}}^s(\log D)$ in grading 0. Then this complex splits into a direct sum:
$$\check{\mathcal{C}}_{s,\log}^{\bullet} = \bigoplus_{J \subset \{1, \ldots, m\},\;\#J=s}\check{\mathcal{C}}_{\log}^{\bullet}\lp\lp\frac{\dd x}{x}\rp_J\rp$$
where $\check{\mathcal{C}}_{\log}^{\bullet}\lp\lp\frac{\dd x}{x}\rp_J\rp$ is the invertible sheaf of the form $\lp\frac{\dd x}{x}\rp_J$ and its section space in chart corresponding to a cone $(I, I_1, \ldots, I_k)$ is either trivial (i.e. $=0$) or 1-dimensional. And it is 1-dimensional if and only if $J\cap I = \emptyset$.
\end{proposition}
\begin{proof}
Let us show that every 1-form with at worst logarithmic singularities in $D$ in a chart $U$ corresponding to a cone $\sigma = (I, I_1, \ldots, I_k)$ lies in a linear hull: $\left\langle \frac{\dd x_i}{x_i}\;|\;i\notin I \right\rangle$. The torus-invariant components of $D \cap U$ are the ones corresponding to rays $I_1, \ldots, I_k$. Now consider generating vectors of the rays of $\sigma$ completed up to a basis in $N$. Those vectors are: $u_1 = e_{i_1}, \ldots, u_s = e_{i_s}$ where $\{i_1, \ldots, i_s\} = I$, $u_{s+1}, \ldots, u_{s+k}$ -- the generators of the rays corresponding to $I_j, j=1, \ldots, k$ and $u_{s+k+1}, \ldots, u_{n}$ -- additional vectors so that the whole collection is a $\mathbb{Z}$-basis of $N$.

Now consider a dual basis in $M$: $\langle \check{u}_i, u_j\rangle = \delta_i^j$. Note that $\frac{\dd \chi^{\check{u}_i}}{\chi^{\check{u}_i}}$ for $i=s+1, \ldots, s+k$ are the forms with logarithmic singularity in the corresponding $I_i$ ray. And the forms $\frac{\dd \chi^{\check{u}_i}}{\chi^{\check{u}_i}}$ for $i>k+s$ are holomorphic in $U$. Since $u_j = e_{i_j}$ for $j\leqslant s$. Without loss of generality one may consider $I = \{1, \ldots, s\}$. Then for $j>s$ one can write $\check{u}_j$ in coordinates of standard basis of $M$: $\check{u}_j = (0, \ldots, 0, \alpha_{j,s+1}, \ldots, \alpha_{j,n})$. Since $\check{u}_{s+1}, \ldots, \check{u}_{n}$ are $n-s$ vectors which can be completed to be a $\mathbb{Z}$-basis of $M$ and all have 0 in first $s$ coordinates, they form a basis in $\mathbb{Z}\left\langle \frac{\dd x_i}{x_i},\;i\notin I\right\rangle$. Remark that this works fine for both ''types'' of cones: with the ray $\mathbf{-1}$ and without it.

Therefore all forms in 
$$\mathbb{Z}\left\langle \frac{\dd x_i}{x_i}\;|\;i\notin I \right\rangle$$
have at worst logarithmic singularities in $D$ in a chart $U$. Now any form with $\frac{\dd x_i}{x_i}, i \in I$ as a component (summand) has a singularity in $\{x_i=0\}$, but $U\cap \UKC \cap \{x_i=0\} \neq \emptyset$ since $I \in \K$, so this singularity is not in $D$. Thus such a form is not a log-form in chart $U$.

It directly follows that higher log-forms ($s$-forms) in $U$ are exactly $\left\langle\lp\frac{\dd x}{x}\rp_J,\; J\cap I = \emptyset,\;\#J=s\right\rangle$

Thus we have a similar situation as in the previous lemma: each component of a \v{C}ech complex splits as the described above direct sum, but this splitting is compatible with the differential, which proves the statement of the proposition.
\end{proof}

\begin{theorem}\label{theorem hodge filtration}
   $$\check{\mathcal{C}}_{s,\log}^{\bullet} \cong \check{\mathcal{C}}_s^{\bullet} \cong \bigoplus_{[m]\in [M], \#I(m)=s}\mathcal{C}_{[m]}^{\bullet}$$
    where $\mathcal{C}_{[m]}^{\bullet}$ is the \v{C}ech complex for the structure sheaf in grading $m$ from Section \ref{structural sheaf coh} and $\cong$ means quasi-isomorphism.
\end{theorem}
\begin{proof}
    
    Recall that the condition $J \cap I = \emptyset$ from Proposition \ref{prop holomorphic sections} for $\check{\mathcal{C}}^i\left(\frac{\mathrm{d}x_J}{x_J}\right)$ to have 1-dimensional (nontrivial) sections space is the same as the condition from Proposition \ref{prop log sections} for $\lp\frac{\dd x}{x}\rp_J$ to have 1-dimensional (nontrivial) sections space as the $\log$-form on $\overline{\UKC}$ in the chart corresponding to any cone of kind $(I, I_1, \ldots, I_k)$. So by Remark \ref{rem cech equivalence} all charts corresponding to cones intersecting with $\Sigma$ by the cone $(I)$ are equivalent for the sheaf $\Omega^s_{\overline{\UKC}}(\log D)$. Choosing in each such equivalence class a representative $(I)$ (i.e. $k=0$) we obtain that $\check{\mathcal{C}}_{\log}^{\bullet}\lp\lp\frac{\dd x}{x}\rp_J\rp$ is homotopy equivalent to $\check{\mathcal{C}}\lp\lp\frac{\dd x}{x}\rp_J\rp$. Thus the first stated in the theorem quasi-isomorphism holds:
    $$\check{\mathcal{C}}_{s,\log}^{\bullet} \cong \check{\mathcal{C}}_s^{\bullet}$$
    
    To prove the second equality, it remains only to show that 
    $$\check{\mathcal{C}}^{\bullet}\left(\frac{\mathrm{d}x_J}{x_J}\right) = \mathcal{C}_{[m]}^{\bullet}$$ for $I(m) = J$.
    
    As already noted, on the chart corresponding to the set of vertices $I$, this sheaf has a nonzero section if and only if $J \cap I = \emptyset$. In turn, the complex $\mathcal{C}_{[m]}^{\bullet}$ has section space $\mathbb{C}$ precisely for those charts $I$ where $I \cap I(m) = \emptyset$ (see remark \ref{section of cech complex rem}). That is, both complexes are \v{C}ech complexes of certain invertible sheaves, and on every affine chart of the covering their section spaces coincide, so the complexes are isomorphic.
\end{proof}

\begin{corollary}\label{filtration cor}
    The associated graded pieces of Hodge filtration can be expressed as follows:
    $$Gr_F^tH^i\lp\UKC, \mathbb{C}\rp \cong \bigoplus_{[m]\in [M], \#I(m)=t}H^{i-t}(\UKC, \mathcal{O}_{\UKC})[m]\cong H^{i-2t, 2t}(\mathcal{Z}_{\K}, \mathbb{C})$$
\end{corollary}    
\begin{proof}
	Since by the first equality of Theorem \ref{theorem hodge filtration} stupid filtration induces Hodge filtration, then $$Gr_F^tH^i\lp\UKC, \mathbb{C}\rp \cong \bigoplus_{[m]\in [M], \#I(m)=t}H^{i-t}(\UKC, \mathcal{O}_{\UKC})[m]$$ by applying the second statement of Theorem \ref{theorem hodge filtration}.
    Then by Theorem \ref{main theorem structural sheaf} the second stated equality holds.
\end{proof}

\subsection{Bigrading and multigrading in the cohomology of $\UKC$}

From \cite{Weber2004} it is known that in the cohomology of toric varieties the mixed Hodge structure has a fairly simple form: in the graded piece of weight $2w$ there is only a component of Hodge type $(w,w)$, and the graded piece of weight $2w+1$ is zero. Combining this with the corollary above, we obtain an alternative proof of Eliyashev's theorem:

\begin{theorem}\label{weight theorem}
Let $\UKC$ be the complement of coordinate subspaces constructed from a simplicial complex $\K$. Let $H^{*,*}(\mathcal{Z}_{\K}, \mathbb{C})$ be the bigraded cohomology of $\mathcal{Z}_{\K}$ in the sense of Theorem \ref{theorem_buchstaber_panov}. Then the graded pieces of the weight filtration on the cohomology of $\UKC$ can be expressed as follows:
    $$Gr^W_wH^{w-s}\lp\UKC, \mathbb{C}\rp = H^{-s, w}(\mathcal{Z}_{\K}, \mathbb{C})$$
    And the graded pieces of Hodge filtration can be expressed as follows:
    $$Gr_F^tH^{2t-s}\lp\UKC, \mathbb{C}\rp \cong H^{-s, 2t}(\mathcal{Z}_{\K}, \mathbb{C})$$
\end{theorem}
\begin{proof}
Since the only non-zero components of $Gr^W_{2w}H^k\lp \UKC, \CC\rp$ are of Hodge type $(w,w)$, Hodge decomposition of $Gr^W_{2w}H^k\lp \UKC, \CC\rp$ is trivial:
$$Gr^W_{2w}H^k\lp \UKC, \CC\rp = F^wGr^W_{2w}H^k\lp \UKC, \CC\rp \cap \overline{F^wGr^W_{2w}H^k\lp \UKC, \CC\rp}$$
Since $F^wGr^W_{2w}H^k\lp \UKC, \CC\rp  \subseteq Gr^W_{2w}H^k\lp \UKC, \CC\rp$ it should coincide with $Gr^W_{2w}H^k\lp \UKC, \CC\rp$, so 
$Gr^W_{2w}H^k\lp \UKC, \CC\rp = F^wGr^W_{2w}H^k\lp \UKC, \CC\rp$,
therefore
$F^{w-1}Gr^W_{2w}H^k\lp \UKC, \CC\rp =$ \\$= Gr^W_{2w}H^k\lp \UKC, \CC\rp$
and since
$F^{w-1}Gr^W_{2w}H^k\lp \UKC, \CC\rp \cap \overline{F^{w+1}Gr^W_{2w}H^k\lp \UKC, \CC\rp} = 0$,
then
$F^{w+1}Gr^W_{2w}H^k\lp \UKC, \CC\rp = 0$.

Now recall (see 3.2.1 from \cite{cattani2014hodge}) that $F^{s}Gr^W_{w} H^k = \frac{F^s H^k \cap W_w H^k}{F^s H^k \cap W_{w-1} H^k}$, so (by induction)
$$F^s H^k(\UKC, \mathbb{C}) \cong \bigoplus_{w \in \mathbb{N}}F^s Gr^W_w H^k(\UKC, \mathbb{C}) \cong \bigoplus_{w \geqslant 2s} Gr^W_w H^k(\UKC, \mathbb{C})$$
Now taking associated graded pieces we obtain
$$Gr^W_{2s}H^k(\UKC, \mathbb{C}) \cong Gr_F^sH^k(\UKC, \mathbb{C})$$

     From Corollary \ref{filtration cor} we know that
    $$Gr^W_{2s}H^k\lp\UKC, \mathbb{C}\rp \cong Gr_F^sH^k\lp\UKC, \mathbb{C}\rp \cong H^{k-2s, 2s}(\mathcal{Z}_{\K}, \mathbb{C})$$
    And also by \cite{Weber2004} $Gr^W_{2s+1} H^*(\UKC) = 0 = H^{*, 2s+1}(\mathcal{Z}_{\K}, \mathbb{C})$, which completes the proof.
\end{proof}

Applying Theorem \ref{cohomology from complex thm}, we obtain the main result of this work:

\begin{theorem}[On the coincidence of multigradings]\label{main theorem}
Let $\K$ be a simplicial complex on $n$ vertices, $\mathcal{Z}_{\K}$ the corresponding moment-angle complex, and $\UKC$ the corresponding complement of coordinate subspaces. For $I\subseteq \{1, \ldots, n\}$, $\#I = t$ and $0\leqslant s\leqslant t$,
let $H^{-s, 2t}(\mathcal{Z}_{\K},\mathbb{C})[I]$ be the cohomology component in multigrading $I$ (the multigrading is induced by the multidegree grading on the Stanley---Reisner ring).

Let $Gr^W_{2t}H^{2t-s}(\UKC, \mathbb{C})[I]$ be the component of the associated graded piece of the weight filtration in multigrading $I$ (the grading arises from the decomposition of the \v{C}ech complex for $\Omega^{t}_{\overline{\UKC}}(\log D)$ into a direct sum of complexes for $\lp\frac{\mathrm{d}x}{x}\rp_I,\;\#I=t$). Then 
$$H^{-s, 2t}(\mathcal{Z}_{\K},\mathbb{C})[I] \cong Gr^W_{2t}H^{2t-s}(\UKC, \mathbb{C})[I]$$
\end{theorem}
\begin{proof}

The multidegree grading on the Stanley---Reisner ring induces a multigrading on the values of the $\Tor$ functor, that is, on $H^{-s, 2t}(\mathcal{Z}_{\K}, \mathbb{C})$. By Theorem \ref{cohomology from complex thm}, this multigrading is expressed as:
$$H^{-s, 2t}(\mathcal{Z}_{\K}, \mathbb{C})[I] = \widetilde{H}^{t-s-1}(\K|_{I}, \mathbb{C})$$

By proposition \ref{prop log sections} and theorem \ref{theorem hodge filtration} \v{C}ech complex for $\Omega^{t}_{\overline{\UKC}}(\log D)$ splits into components indexed by $I\subseteq \{1, \ldots, n\}$ where every component is isomorphic to the \v{C}ech complex for $\mathcal{O}_{\UKC}[m],\;I(m)=I$. At the same time by theorem \ref{weight theorem} such components with $\#I=t$ in total compute graded piece of weight $2t$. And by theorem \ref{main theorem structural sheaf} the cohomology of each component is exactly $\widetilde{H}^{t-s-1}(\K|_{I}, \mathbb{C})$. Combining all together we obtain the claimed statement:

$$Gr^W_{2t}H^{2t-s}(\UKC, \mathbb{C})[I] = H^{2t-s}(\UKC, \mathbb{C})[I] = $$$$ = H^{t-s}(\UKC, \mathcal{O}_{\UKC})[m] = \widetilde{H}^{t-s-1}(\K|_{I(m)}, \mathbb{C})$$ 
\end{proof}

\subsection{Application of mixed Hodge structures to the computation of the number of points on $\UKQ$}

The work of Pierre Deligne and Grothendieck \cite{grothendieck66} allows us to relate the weight filtration on cohomology of $\UKC$ to the number of points on $\UKQ$. Below we recall the known techniques that allow us to find the number of $\F_q$-rational points on a toric variety from the weight filtration on the étale cohomology of that variety, and then apply this to $U_{\K}$.

From \cite{Weber2004} it is known that the Frobenius endomorphism acts on the cohomology of a toric variety of weight $2w$ by multiplication by $q^{w}$. Therefore, for a toric variety, the Lefschetz fixed-point theorem over a finite field is formulated as follows:

\begin{theorem}[Lefschetz fixed point theorem over finite fields for toric varieties]
    $$\#V_{\Sigma}(\mathbb{F}_q) = \sum_{k}\lp(-1)^k\sum_w q^w\cdot\dim Gr^W_{2w} H^k_{c,\text{\'et}}(V_{\Sigma})\rp$$
    where $H^k_{c,\text{\'et}}$ is compactly supported étale cohomology.
\end{theorem}

In addition, one of the most important tools for working with étale cohomology is the comparison theorem (see Milne's book \cite{milne1980}):
\begin{theorem}[Comparison of étale cohomology with singular cohomology]
Let $X$ be a smooth algebraic variety. Then 

$$\forall i, w\;\dim Gr^W_wH^i(X(\mathbb{C}), \mathbb{C}) = \dim Gr^W_wH^i_{\text{\'et}}(X, \mathbb{Q}_l)$$

where $X(\mathbb{C})$ is the set of $\mathbb{C}$-points of $X$, equipped with the complex-analytic topology.
\end{theorem}

There is also the duality theorem between ordinary cohomology and cohomology with compact supports, sometimes also called Poincaré duality (see Milne's book \cite{milne1980}):

\begin{theorem}[Duality between cohomology and cohomology with compact supports]\label{etale duality theorem}
For $X$ a smooth algebraic variety of dimension $d$, there is an isomorphism of Galois modules:

$$H^i_{c, \text{\'et}}(X, \mathbb{Q}_l) \cong H^{2d-i}_{\text{\'et}}(X, \mathbb{Q}_l)^{\check{}}(-d)$$

where $M(-d)$ is the Tate twist of the Galois module $M$.
\end{theorem}

In particular, since the Frobenius endomorphism is precisely the action of the Galois group, the duality theorem has a corollary about the action of the Frobenius operator on étale cohomology with compact support:

\begin{corollary}
Let $\pi$ be the Frobenius operator. Then

$$\tr\lp\pi|_{Gr^W_w H^{k}_{c,\text{\'et}}(X)}\rp = \tr\lp\pi|_{Gr^W_{2d-w}H^{2d-k}_{\text{\'et}}(X)}\rp$$
\end{corollary}

From this we obtain the theorem:

\begin{theorem}[Grothendieck---Lefschetz formula for toric varieties]\label{thm grothendieck-lefshetz}
Let $X$ be a smooth toric variety of dimension $d$. Then 

$$\#X(\mathbb{F}_q) = \sum_{s \in \mathbb{N}}q^s\cdot\lp\sum_{k\in\mathbb{N}}(-1)^k\dim Gr^W_{2(d-s)}H^{k}(X, \mathbb{C})\rp$$
\end{theorem}

\begin{proof}
After applying the duality theorem and the comparison theorem to the Grothendieck---Lefschetz formula, we obtain the formula:

$$\#X(\mathbb{F}_q) = \sum_{s \in \mathbb{N}}q^s\cdot\lp\sum_{k'\in\mathbb{N}}(-1)^{k'}\dim Gr^W_{2(d-s)}H^{2d-k'}(X, \mathbb{C})\rp$$

It remains only to relabel $k=2d-k'$ and note that $(-1)^k = (-1)^{k'}$.
\end{proof}

Combining the theorem above with Theorem \ref{weight theorem}, we obtain:

\begin{theorem}\label{theorem points count cohomology}
    Let $\K$ be a simplicial complex on $n$ vertices, let $\UKQ$ be the corresponding complement to coordinate subspaces over $\F_q$ and let $\mathcal{Z}_{\K}$ be the corresponding moment-angle complex. Then:
    $$\#\UKQ = \sum_{t = 0}^n a_tq^t,\;\text{where }a_t = \sum_{s=0}^{n-t}(-1)^{s}\cdot\dim H^{-s,2(n-t)}(\mathcal{Z}_{\K})$$
    where $H^{-s, 2(n-t)}(\mathcal{Z}_{\K})$ is the bigraded component of the cohomology of $\mathcal{Z}_{\K}$ in the sense of Theorem \ref{theorem bigraded components BP}.
\end{theorem}
\begin{proof}
    Theorem \ref{weight theorem} says that $Gr^W_{2(n-t)}H^{2n - 2t - s}(\UKC, \mathbb{C}) = H^{-s, 2(n-t)}(\mathcal{Z}_{\K})$, and since $(-1)^{2n - 2t - s} =$\\$= (-1)^s$, the formula from Theorem \ref{thm grothendieck-lefshetz} turns into exactly the formula in the statement.
\end{proof}

In fact, the statement of Theorem \ref{theorem points count cohomology} can also be proved without applying "heavy machinery" such as the comparison theorem for étale cohomology and the Lefschetz fixed point theorem, using only combinatorial considerations. This is done in the next section. 
\section{Arithmetic properties of the scheme $U_{\K}$}\label{finite fields}
\subsection{Combinatorial formulas for the number of points $\UKQ$}
 The algebraic variety $U_{\K}$ is defined over $\mathbb{Z}$; in other words, it is a scheme (of finite type by definition) over $\Spec\mathbb{Z}$. This observation allows us to apply ideas from arithmetic geometry to $\UKQ$, in particular, to study the number of points over finite fields and the global $\zeta$-function.

To count the points of $\UKQ$, we use Remark \ref{remark_union}:
$$\UKQ = \mathbb{A}_n(\F_q)\setminus \bigsqcup_{I \notin \K}\{(x_1,\ldots,x_n) \in \mathbb{A}_n(\F_q)\;\mid\;\forall i\in I\;x_i=0,\;\forall j\notin I\;x_j\neq 0\}$$
This immediately gives a formula for the number of points in $\UKQ$:
$$\#\UKQ = q^n - \sum_{I\notin\K}(q-1)^{n-\#I}$$
Using the fact that $q^n = \sum_{k=0}^n {n\choose k}(q-1)^k = \sum_{k=0}^n {n\choose n-k}(q-1)^k$ and grouping the $I$ in the preceding sum according to $\#I$, we obtain:
$$\#\UKQ = \sum_{k=0}^{n}(q-1)^{n-k}\cdot\left({n\choose k}-\#\{I \notin \K\;\mid\;\#I=k\}\right) = \sum_{k=-1}^{d}(q-1)^{n-k-1}f_k(\K)$$
the last equality holds because $$f_{k-1}(\K) = \#\{I\in\K\;\mid\;\#I=k\} = {n\choose k} - \#\{I\notin \K\;\mid\;\#I=k\}$$
It follows that:

\begin{proposition}\label{prop number of points torus}
$$\#\UKQ = (q-1)^{n}\sum_{k=-1}^d\frac{f_k(\K)}{(q-1)^{k+1}}$$
\end{proposition}

Recall the definition of Hilbert---Poincaré series for a graded $\CC$-algebra $A$:
\begin{definition}
Let $A$ be a graded $\CC$-algebra with decomposition into graded components 
$$A = \bigoplus_{k \geqslant 0}A_k$$
Then the Hilbert---Poincaré series of $A$ is the following:
$$H_A(t) = \sum_{k\geqslant 0}\dim A_k \cdot t^k$$
\end{definition}

It is known (see Chapter II of \cite{stanley1996combinatorics}) that the Hilbert--Poincaré series of Stanley--Reisner ring can be expressed in terms of the $f$-vector of the complex $\K$:
$$H_{\mathbb{C}[\K]}(t) = \sum_{k=-1}^df_k(\K)\cdot\lp\frac{t}{1-t}\rp^{k+1}$$
substituting $\frac{t}{1-t}=\frac{1}{q-1} \Leftrightarrow t=\frac{1}{q}$, we obtain:
\begin{theorem}\label{prop hilbert poincare}
    Let $\K$ be a simplicial complex on $n$ vertices, and let $\UKQ$ be the corresponding coordinate subspace arrangement complement over $\F_q$. Let $H_{\mathbb{C}[\K]}(t)$ be the Hilbert--Poincaré series of the Stanley--Reisner ring of the complex $\K$. Then the number of points of $\UKQ$ can be expressed as follows:
$$\#\UKQ = (q-1)^n\cdot H_{\mathbb{C}[\K]}\lp\frac{1}{q}\rp$$
\end{theorem}

The number of points of $U_{\K}$ over the finite field $\F_q$ is a polynomial in $q$. The coefficients of this polynomial can be readily obtained from the formulas above:

\begin{theorem}\label{thm combinatorial number of points}
Let $\K$ be a simplicial complex on $n$ vertices, and suppose that the maximal dimension of a simplex in $\K$ is $d$. Then the number of points of $U_{\K}$ over $\F_q$ can be expressed as follows:

$$\#\UKQ = \sum_{k=0}^na_kq^k,\;a_k = \sum_{j=-1}^df_j(\K)\cdot{n-j-1\choose k}\cdot(-1)^{n-j-k-1}$$
where all undefined binomial coefficients are taken to be 0.
\end{theorem}

\begin{proof}
$$\#\UKQ = \sum_{k=-1}^df_k(\K)\cdot(q-1)^{n-k-1}$$
Accordingly, after expanding the parentheses, the coefficient of the monomial $q^k$ is:

$$a_k = \sum_{j=-1}^d f_j(\K)\cdot (-1)^{(n-j-1) - k}\cdot {n-j-1\choose k}$$
\end{proof}

In the case where $\K$ is a boundary of a polygon its $f$-vector has simple form $(1,n,n)$. Then one can give a more elegant formula for $a_k$:
$$a_k = (-1)^{n-k}\left({n\choose k}-n{n-1\choose k}+n{n-2\choose k}\right) =$$
$$= (-1)^{n-k}\left({n\choose k} - n\left({n-1\choose k}-{n-2\choose k}\right)\right) = (-1)^{n-k}\left({n\choose k}-n{n-2\choose k-1}\right)$$
\begin{remark}
For $\K$ equal to the boundary of a polygon, $a_k$ is graded-symmetric with respect to the substitution $k\mapsto n-k$ (that is, under this substitution it is multiplied by $(-1)^n$), and:
$$(-1)^na_0 = a_n = 1,\;a_1=a_{n-1}=0,\;(-1)^{n-k}a_k<0\text{ для } 2\leqslant k\leqslant n-2$$
\end{remark}
\begin{proof}
    The symmetry follows from the formula: binomial coefficients are symmetric, $${n\choose k}={n\choose n-k},\;{n-2\choose k-1}={n-2\choose n-2-(k-1)}={n-2\choose (n-k)-1}$$ 
    and $(-1)^{n-k}=(-1)^k\cdot (-1)^n$.
    $$(-1)^na_0=a_n=(-1)^0\left( 1+n\cdot 0\right) = 1$$
    $$(-1)^na_1=a_{n-1}=(-1)^1\left( n - n\cdot 1\right) = 0$$
    $${n\choose k} - n{n-2\choose k-1} = \frac{n\cdot(n-2)!}{(n-k-1)!(k-1)!}\left( \frac{n-1}{(n-k)k} - 1\right) < 0\text{ при } 2\leqslant k\leqslant n-2$$
    The inequality holds because the left factor is always positive, while the expression in parentheses is negative when $\min(k,n-k)>1$.
\end{proof}

\begin{remark}
In terms of Cox's quotient construction of a toric variety, the case when $\K$ is the boundary of an $n$-gon corresponds to the construction of a smooth projective toric surface $X$ with divisor class group $Cl(X)\cong\mathbb{Z}^{n-2}$.
\end{remark}

Finally, using Theorem \ref{thm combinatorial number of points}, we can obtain a combinatorial proof of Theorem \ref{theorem points count cohomology} from the previous section:

\begin{proof}[Combinatorial proof of Theorem \ref{theorem points count cohomology}]
Denote $H^{-s,2t}(\mathcal{Z}_{\K}, \mathbb{C}) = H^{-s, 2t}$. We use Theorem \ref{theorem_buchstaber_panov} to compute $\dim H^{-s, 2t}$. 

    Note that, as a differential graded vector space, $\mathcal{A}(\K)$ decomposes as a direct sum:
    $$\mathcal{A}(\K) = \bigoplus_{t=0}^n \mathcal{A}(\K)^{*,2t}$$
    since the differential has bidegree $(1,0)$.
    
    Note that $a_t$ in the statement of Theorem \ref{theorem points count cohomology} is precisely the Euler characteristic of $H^{*, 2(n-t)}$. In particular, $$a_t = \sum_{s=0}^{n-t} (-1)^s\cdot\dim \mathcal{A}(\K)^{-s, 2(n-t)}$$
    
    By definition, $\mathcal{A}(\K)^{-s, 2(n-t)}$ is generated by monomials of the form $v^I\otimes u^J$, where $I,J\subseteq \{1, \ldots, n\}$, $I\cap J = \emptyset,\;I \in \K,\;\#J = s,\;\#I+\#J = n-t$. Accordingly, the number of generators is the number of such $I$ multiplied by the number of such $J$:
    $$\dim \mathcal{A}(\K)^{-s, 2(n-t)} = f_{n-t-s-1}(\K)\cdot{n-(n-t-s)\choose s}$$
    Thus, the resulting formula for $a_t$ is:
    
    $$a_t = \sum_{s=0}^t(-1)^s\cdot f_{n-t-s-1}(\K)\cdot {t+s\choose s}$$
	Note that the $(n-t-s-1)$-dimensional simplices are precisely those $I \in \K$ such that $\#I = n-t-s$. 
    
    After the substitution $j = n-t-s-1,\;s = n-t-j-1$, we obtain:
    $$a_t = \sum_{j=-1}^d f_j(\K)\cdot {n-j-1 \choose t}\cdot (-1)^{n-j-t-1}$$
    where $d$ is the maximal dimension of a simplex in $\K$, and all undefined binomial coefficients are taken to be zero.
    
    Theorem \ref{thm combinatorial number of points} states precisely that the number of points $\#\UKQ$ is expressed as $\sum_{t=0}^n a_tq^t$ for these $a_t$, thereby completing the combinatorial proof of Theorem \ref{theorem points count cohomology}.
\end{proof}

\subsection{Global $\zeta$-functions of the scheme $U_{\K}$}

Knowing the number of points over different finite fields allows us to express the global $\zeta$-function of $U_{\K}$ (see chapter 6 in \cite{maninpanchishkin} for reference on $\zeta$-functions for schemes over $\Spec \Z$).

Recall the following property of global $\zeta$-functions:
\begin{proposition}\label{zeta polynom}
If $X$ is a scheme of finite type over $\Spec\Z$ and $\#X(\F_q) = \sum_{i=0}^k a_iq^i$, then
$$\zeta_X(s) = \prod_{i=0}^k \zeta(s-i)^{a_i}$$
\end{proposition}
\begin{proof}
Decompose $\zeta_X(s)$ as a product over the primes $p$:
$$\zeta_X(s) = \prod_{p}Z_p(X, p^{-s})$$
where 
$$Z_p(X, t) = \exp\lp\sum_{j\geqslant 1} \#X(\F_{p^j})\frac{t^j}{j}\rp$$

$$Z_p(X, t) = \prod_{i=0}^k\exp\lp a_i\sum_{j \geqslant 1} p^{ij}\frac{t^j}{j}\rp = \prod_{i=0}^k \lp\frac{1}{1-p^{i-s}}\rp^{a_i}$$
Therefore
$$\zeta_X(s) = \prod_{p}\prod_{i=0}^k \lp\frac{1}{1-p^{i-s}}\rp^{a_i} = \prod_{i=0}^k\zeta(s-i)^{a_i}$$
\end{proof}

Applying the formula for the number of points of $\UKQ$ to Proposition \ref{zeta polynom}, we obtain the following theorem:

\begin{theorem}[Expression for the global $\zeta$-function of a coordinate subspace arrangement complement]\label{thm zeta}
Let $\K$ be a simplicial complex on $n$ vertices, with maximal simplex dimension $d$, and let $U_{\K}$ be the corresponding coordinate subspace arrangement complement, $\mathcal{Z}_{\K}$ be the corresponding moment-angle complex. Then the global $\zeta$-function of $U_{\K}$ can be expressed as follows:

$$\zeta_{U_{\K}}(s) = \prod_{k=0}^n\zeta(s-k)^{a_k}, \text{ where } \#\UKQ = \sum_{k=0}^n a_kq^k$$
That is, 

$$\zeta_{U_{\K}}(s) = \prod_{w =0}^n\zeta(s + w - n)^{\sum_{t=0}^d (-1)^t\dim H^{-t,2w}(\mathcal{Z}_{\K})}$$
\end{theorem}

\begin{corollary}
In the case where $\K$ is the boundary of an $n$-gon, the global $\zeta$-function of $U_{\K}$ is expressed as follows:

$$\zeta_{U_{\K}}(s) = \prod_{k=0}^n\zeta(s-k)^{(-1)^{n-k}\left({n\choose k}-n{n-1\choose k}+n{n-2\choose k}\right)}$$
\end{corollary}
 
\section{Examples}\label{examples}
In this section we consider several examples that are important for understanding the topics described above. The examples are titled with respect to toric varieties obtained from the corresponding $U_{\K}$ in Cox's quotient construction.

\subsection{Projective space}\label{ex projective space}

Projective space $\mathbb{P}^n$ is the $n$-dimensional toric variety with a fan having $n+1$ rays: $n$ rays are generated by the standard basis vectors, and the $(n+1)$-th ray is generated by the sum of the basis vectors taken negative. Furthermore, any proper subset of the set of rays forms a cone of this fan.

The intersection of the fan with the unit sphere gives a triangulation of the $(n-1)$-dimensional sphere with $n+1$ vertices, that is, the boundary of the standard $n$-dimensional simplex. Denote it by $\partial \Delta^n$.

Consider $\K = \partial \Delta^n$. The only $I\notin \K$ is $I = \{1,\ldots, n+1\}$, hence
$$U_{\partial \Delta^n} = \mathbb{A}^{n+1}\setminus \{0\}$$

The one-dimensional torus acts on $U_{\partial\Delta^n}$ by diagonal multiplication:
$$\mathbb{T}\acts U_{\partial\Delta^n},\;\mathbb{T}\ni \alpha, (z_1, \ldots, z_{n+1}) \xmapsto{\alpha} (\alpha z_1, \ldots, \alpha z_{n+1})$$

In the complex case $\mathbb{T}(\mathbb{C}) = \mathbb{C}^{\times} \cong S^1\times\mathbb{R}_{>0}^{\times}$, and then

$$\mathcal{Z}_{\partial\Delta^n} = U_{\partial\Delta^n}(\mathbb{C})/\mathbb{R}_{>0}^{\times} \cong S^{2n+1}$$

This gives rise to the Hopf fibration:
$$S^1 \hookrightarrow \mathcal{Z}_{\partial\Delta^n} \rightarrow \mathbb{P}^n = U_{\partial\Delta^n}(\mathbb{C})/\mathbb{C}^{\times}$$

The cohomology of the sphere is well known:
$H^0(S^{2n+1}, \mathbb{C}) = \mathbb{C}, \; H^{2n+1}(S^{2n+1}, \mathbb{C}) = \mathbb{C}$, and the rest are zero. In particular, the cohomology of $U_{\partial\Delta^n}(\mathbb{C})$ is arranged in the same way.

To find the decomposition of $H^0(U_{\partial\Delta^n}(\mathbb{C}), \mathbb{C})$ and $H^{2n+1}(U_{\partial\Delta^n}(\mathbb{C}), \mathbb{C})$ into bigraded components in the sense of Buchstaber---Panov, let us construct a resolution of the Stanley---Reisner ring: $\mathbb{C}[\partial\Delta^n] = \mathbb{C}[v_1, \ldots, v_{n+1}]/(v_1\cdot\ldots\cdot v_{n+1})$.

In this case it is easy to write it down directly; the following sequence is obviously exact:

$$0\rightarrow \mathbb{C}[v_1, \ldots, v_{n+1}] \xrightarrow{1\mapsto v_1\cdot\ldots\cdot v_{n+1}}\mathbb{C}[v_1, \ldots, v_{n+1}] \rightarrow \mathbb{C}[\partial\Delta^n] \rightarrow 0$$

Recall the statement of Theorem \ref{cohomology from complex thm}:

$$\Tor^s_{\mathbb{C}[v_1, \ldots, v_{n+1}]}(\mathbb{C}[\partial\Delta^n], \mathbb{C})[I] \cong \widetilde{H}^{\#I - s-1}(\partial\Delta^n|_{I}, \mathbb{C})$$ 

In the case $I=\{1, \ldots, n+1\},\;\partial\Delta^n|_{I} \cong S^{n-1}$, that is, the only nontrivial reduced cohomology is in grading $n-1 = \#I-1-1$, which corresponds to the nontrivial $\Tor^1_{\mathbb{C}[v_1, \ldots, v_{n+1}]}(\mathbb{C}[\partial\Delta^n], \mathbb{C})$.

In the case $I=\emptyset,\;\partial\Delta^n|_{I} = \emptyset$ has nontrivial reduced cohomology in grading $-1 = \#I-0-1$, which corresponds to the nontrivial $\Tor^0_{\mathbb{C}[v_1, \ldots, v_{n+1}]}(\mathbb{C}[\partial\Delta^n], \mathbb{C})$. 

For other $I$, $\partial\Delta^n|_{I}$ is contractible, so it has no nontrivial reduced cohomology.

From Theorem \ref{weight theorem} we obtain that
$$Gr_W^0 H^{0}(U_{\partial\Delta^n}(\mathbb{C}), \mathbb{C}) = \mathbb{C},\;Gr_W^{2(n+1)} H^{2n+1}(U_{\partial\Delta^n}(\mathbb{C}), \mathbb{C}) = \mathbb{C}$$
and all other weight components are zero.

When passing to étale cohomology with compact supports, by Theorem \ref{etale duality theorem} the weight and grading change to the complement with respect to the real dimension, that is, $2(n+1)$:

$$\dim Gr_W^{2(n+1)} H_{c,\text{\'et}}^{2(n+1)}(U_{\partial\Delta^n}) = \dim Gr_W^0 H^{0}(U_{\partial\Delta^n}(\mathbb{C}), \mathbb{C}) = 1
$$
$$\dim Gr_W^{0} H_{c,\text{\'et}}^{1}(U_{\partial\Delta^n}) = \dim Gr_W^{2(n+1)} H^{2n+1}(U_{\partial\Delta^n}(\mathbb{C}), \mathbb{C}) = 1$$

Applying the Grothendieck---Lefschetz trace formula, we get:

$$(-1)^{2(n+1)}\cdot q^{\frac{2(n+1)}{2}}\cdot \dim Gr_W^{2(n+1)} H_{c,\text{\'et}}^{2(n+1)}(U_{\partial\Delta^n}) + (-1)^1\cdot q^{\frac{0}{2}}\cdot \dim Gr_W^{0} H_{c,\text{\'et}}^{1}(U_{\partial\Delta^n}) =$$
$$= q^{n+1}-1 = \#U_{\partial\Delta^n}(\F_q)$$

Note that the answer turned out to be exactly what one would expect: $U_{\partial\Delta^n}(\F_q) = \mathbb{F}_q^{n+1}\setminus \{0\}$, so the number of rational points is indeed $q^{n+1}-1$.

The Hilbert---Poincar\'e series for the Stanley---Reisner ring can also be expressed fairly easily in this case:

In $\mathbb{C}[\partial\Delta^n]$ there are no monomials in which the degree of each variable is at least 1; it is easy to count the number of such ''forbidden'' monomials in grading $k$: degree $n+1$ is ''reserved'' for the mandatory occurrence of each variable, and the remaining factors can be chosen arbitrarily. Thus, there are exactly as many such monomials as there are monomials of degree $k-n-1$ in the ring $\mathbb{C}[v_1, \ldots, v_{n+1}]$. That is, the Hilbert---Poincar\'e series for the Stanley---Reisner ring is expressed in terms of the Hilbert---Poincar\'e series for the polynomial ring:
$$H_{\mathbb{C}[\partial\Delta^n]}(t) = H_{\mathbb{C}[v_1, \ldots, v_{n+1}]}(t) - t^{n+1}H_{\mathbb{C}[v_1, \ldots, v_{n+1}]}(t)$$

The Hilbert---Poincar\'e series for the polynomial ring equals:

$$H_{\mathbb{C}[v_1, \ldots, v_{n+1}]}(t) = \sum_{k\geqslant 0}{k+n\choose k}t^k = \frac{1}{(1-t)^{n+1}}$$

Therefore $H_{\mathbb{C}[\partial\Delta^n]}(t) = \frac{1 - t^{n+1}}{(1-t)^{n+1}}$, in particular, 
$$(q-1)^{n+1}\cdot H_{\mathbb{C}[\partial\Delta^n]}\lp\frac{1}{q}\rp = \frac{(q-1)^{n+1}\lp 1-\frac{1}{q^{n+1}}\rp}{\lp 1-\frac{1}{q}\rp^{n+1}} =$$
$$= (q-1)^{n+1}\frac{q^{n+1}-1}{(q-1)^{n+1}} = q^{n+1}-1 = \#\UKQ$$

which illustrates Theorem \ref{prop hilbert poincare}.

\subsection{Projective toric surface with Picard number 3}\label{ex pentagon}

Now let us consider the ''minimal nontrivial example'', namely, the case when $\K$ is the boundary of a pentagon. In terms of the quotient construction this corresponds to a projective toric surface with Picard number 3. 
The correspondence between integral pentagons and projective toric surfaces with Picard number 3 is written in details in \cite{Batyrev1991}.

In this case the Stanley---Reisner ideal $(v^I\;|\;I \notin \K)$ is generated by degree-2 monomials corresponding to the diagonals of the pentagon: $v_1v_3, v_2v_4, v_3v_5, v_4v_1, v_5v_2$. Indeed, $I \notin \K$ can consist of 2, 3, 4, or 5 vertices. If it consists of 2 vertices, then it is a diagonal of the pentagon. Moreover, among any $\geqslant 3$ vertices there will be two forming a diagonal.

Using the decomposition into disjoint tori from Remark \ref{remark_union}, it is easy to count the number of points of $\UKQ$:
$$\#\UKQ = (q-1)^5+5(q-1)^4 + 5(q-1)^3 = q^5 - 5q^3 + 5q^2 - 1$$

To relate this to the cohomology of $\UKC$, let us find a resolution of the Stanley---Reisner ring:
$\mathbb{C}[\K] = \mathbb{C}[v_1, \ldots, v_5]/(v_1v_3, v_2v_4, v_3v_5, v_4v_1, v_5v_2)$
$$0 \rightarrow \mathbb{C}[v_1, \ldots, v_5]\left\langle v^{\{1,2,3,4,5\}}\right\rangle \xrightarrow{\mathrm{d}_3} \mathbb{C}[v_1, \ldots, v_5]\left\langle v^{\{1,3,5\}}_{135}, v^{\{2,4,1\}}_{241}, v^{\{3,5,2\}}_{352}, v^{\{4,1,3\}}_{413}, v^{\{5,2,4\}}_{524}\right\rangle \xrightarrow{\mathrm{d}_2}$$
$$\xrightarrow{\mathrm{d}_2}  \mathbb{C}[v_1, \ldots, v_5]\left\langle v^{\{1,3\}}_{13}, v^{\{2,4\}}_{24}, v^{\{3,5\}}_{35}, v^{\{4,1\}}_{41}, v^{\{5,2\}}_{52}\right\rangle \xrightarrow{\mathrm{d}_1} \mathbb{C}[v_1, \ldots, v_5] \rightarrow \mathbb{C}[\K] \rightarrow 0$$

The lower indices are used to distinguish the various generators of the $\mathbb{C}[v_1, \ldots, v_5]$-module. The upper indices denote the multigrading of the corresponding generators.

The differentials are arranged as follows:
$$d_1: v^{\{i, i+2\}}_{i(i+2)} \mapsto v^iv^{i+2},\; d_2: v^{\{i, i+2, i+4\}}_{i(i+2)(i+4)} \mapsto v^i\cdot v_{(i+2)(i+4)}^{\{i+2, i+4\}} - v^{i+4}\cdot v_{i(i+2)}^{\{i,i+2\}}$$
$$d_3: v^{\{1,2,3,4,5\}} \mapsto v^{\{2,4\}}\cdot v^{\{1,3,5\}}_{135} + v^{\{3,5\}}\cdot v^{\{2,4,1\}}_{241} + v^{\{4,1\}}\cdot v^{\{3,5,2\}}_{352} + v^{\{5,2\}}\cdot v^{\{4,1,3\}}_{413} + v^{\{1,3\}}\cdot v^{\{5,2,4\}}_{524}$$

It is clear that this is indeed a resolution. Note that since all differentials on the generators produce a sum of non-constant monomials multiplied by generators, after tensoring with $\mathbb{C}$ all differentials become zero. So, the resolution tensored with $\mathbb{C}$:
$$0 \rightarrow \mathbb{C}_{\{1,2,3,4,5\}} \rightarrow \bigoplus_{i=1}^5 \mathbb{C}_{\{i,i+2,i+4\}} \rightarrow \bigoplus_{i=1}^5 \mathbb{C}_{\{i,i+2\}} \rightarrow \mathbb{C}_{\emptyset}\rightarrow 0$$

The lower index denotes the multigrading of the corresponding component. Recall that by Theorem \ref{main theorem}
$$H^{-s, 2t}(\UKC, \CC) \cong \bigoplus_{I\subseteq\{1, \ldots, 5\}, \#I = t} \Tor^{s}_{\mathbb{C}[v_1, \ldots, v_5]}(\mathbb{C}[\K], \mathbb{C})[I]$$

Denote $h^{-s, 2t} = \dim H^{-s,2t}(\UKC, \mathbb{C})$, then:
$$h^{0,0} = 1,\; h^{-1, 4} = 5,\; h^{-2, 6} = 5,\; h^{-3, 10} = 1$$

Denote $h_{c,\text{\'et}}^{k,2w} = \dim Gr_W^{2w}H_{c, \text{\'et}}^k(U_{\K})$, then $h^{-s, 2t} = h_{c,\text{\'et}}^{10-(2t-s), 10-2t}$:
$$h_{c,\text{\'et}}^{10,10} = h^{0,0}=1,\;h_{c,\text{\'et}}^{7, 6} = h^{-1,4} = 5,\; h_{c,\text{\'et}}^{6, 4} = 5,\; h_{c,\text{\'et}}^{3, 0}=h^{-3, 10} = 1$$

And, applying the Grothendieck---Lefschetz formula, we get:
$$\#\UKQ = (-1)^{10}\cdot q^{\frac{10}{2}}\cdot h_{c,\text{\'et}}^{10,10} + (-1)^{7}\cdot q^{\frac{6}{2}}\cdot h_{c,\text{\'et}}^{7, 6} + (-1)^6\cdot q^{\frac{4}{2}}\cdot h_{c,\text{\'et}}^{6, 4} + (-1)^{3}\cdot q^{\frac{0}{2}}\cdot h_{c,\text{\'et}}^{3, 0} = $$
$$= q^5 - 5q^3 + 5q^2 - 1$$

The resulting multigradings $I$ can be related to the reduced cohomology $\widetilde{H}^*(\K|_I, \mathbb{C})$:

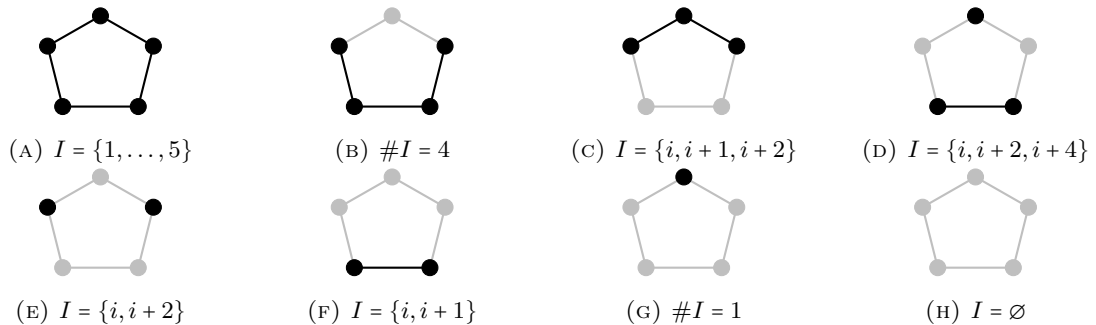
\begin{figure}[H]
\begin{subfigure}{0.24\textwidth}
\begin{center}
\begin{tikzpicture}
\coordinate (A) at (-0.5,0);
\coordinate (B) at (0.5,0);
\coordinate (C) at (0.7,0.8);
\coordinate (D) at (0,1.2);
\coordinate (E) at (-0.7,0.8);

\filldraw (A) circle (3pt);
\filldraw (B) circle (3pt);
\filldraw (C) circle (3pt);
\filldraw (D) circle (3pt);
\filldraw (E) circle (3pt);

\draw[black, thick] (A) -- (B);
\draw[black, thick] (C) -- (B);
\draw[black, thick] (C) -- (D);
\draw[black, thick] (D) -- (E);
\draw[black, thick] (A) -- (E);

\end{tikzpicture}
\end{center}
\caption{$I = \{1,\ldots, 5\}$}
\end{subfigure}
\begin{subfigure}{0.24\textwidth}
\begin{center}
\begin{tikzpicture}
\coordinate (A) at (-0.5,0);
\coordinate (B) at (0.5,0);
\coordinate (C) at (0.7,0.8);
\coordinate (D) at (0,1.2);
\coordinate (E) at (-0.7,0.8);

\draw[black, thick] (A) -- (B);
\draw[black, thick] (C) -- (B);
\draw[gray!50, thick] (C) -- (D);
\draw[gray!50, thick] (D) -- (E);
\draw[black, thick] (A) -- (E);

\filldraw (A) circle (3pt);
\filldraw (B) circle (3pt);
\filldraw (C) circle (3pt);
\filldraw[gray!50] (D) circle (3pt);
\filldraw (E) circle (3pt);

\end{tikzpicture}
\end{center}
\caption{$\#I = 4$}
\end{subfigure}
\begin{subfigure}{0.24\textwidth}
\begin{center}
\begin{tikzpicture}
\coordinate (A) at (-0.5,0);
\coordinate (B) at (0.5,0);
\coordinate (C) at (0.7,0.8);
\coordinate (D) at (0,1.2);
\coordinate (E) at (-0.7,0.8);

\draw[gray!50, thick] (A) -- (B);
\draw[gray!50, thick] (C) -- (B);
\draw[black, thick] (C) -- (D);
\draw[black, thick] (D) -- (E);
\draw[gray!50, thick] (A) -- (E);

\filldraw[gray!50] (A) circle (3pt);
\filldraw[gray!50] (B) circle (3pt);
\filldraw (C) circle (3pt);
\filldraw (D) circle (3pt);
\filldraw (E) circle (3pt);

\end{tikzpicture}
\end{center}
\caption{$I = \{i, i+1, i+2\}$}
\end{subfigure}
\begin{subfigure}{0.24\textwidth}
\begin{center}
\begin{tikzpicture}
\coordinate (A) at (-0.5,0);
\coordinate (B) at (0.5,0);
\coordinate (C) at (0.7,0.8);
\coordinate (D) at (0,1.2);
\coordinate (E) at (-0.7,0.8);

\draw[black, thick] (A) -- (B);
\draw[gray!50, thick] (C) -- (B);
\draw[gray!50, thick] (C) -- (D);
\draw[gray!50, thick] (D) -- (E);
\draw[gray!50, thick] (A) -- (E);

\filldraw (A) circle (3pt);
\filldraw (B) circle (3pt);
\filldraw[gray!50] (C) circle (3pt);
\filldraw (D) circle (3pt);
\filldraw[gray!50] (E) circle (3pt);

\end{tikzpicture}
\end{center}
\caption{$I = \{i, i+2, i+4\}$}
\end{subfigure}

\begin{subfigure}{0.24\textwidth}
\begin{center}
\begin{tikzpicture}
\coordinate (A) at (-0.5,0);
\coordinate (B) at (0.5,0);
\coordinate (C) at (0.7,0.8);
\coordinate (D) at (0,1.2);
\coordinate (E) at (-0.7,0.8);

\draw[gray!50, thick] (A) -- (B);
\draw[gray!50, thick] (C) -- (B);
\draw[gray!50, thick] (C) -- (D);
\draw[gray!50, thick] (D) -- (E);
\draw[gray!50, thick] (A) -- (E);

\filldraw[gray!50] (A) circle (3pt);
\filldraw[gray!50] (B) circle (3pt);
\filldraw (C) circle (3pt);
\filldraw[gray!50] (D) circle (3pt);
\filldraw (E) circle (3pt);

\end{tikzpicture}
\end{center}
\caption{$I = \{i, i+2\}$}
\end{subfigure}
\begin{subfigure}{0.24\textwidth}
\begin{center}
\begin{tikzpicture}
\coordinate (A) at (-0.5,0);
\coordinate (B) at (0.5,0);
\coordinate (C) at (0.7,0.8);
\coordinate (D) at (0,1.2);
\coordinate (E) at (-0.7,0.8);

\draw[black, thick] (A) -- (B);
\draw[gray!50, thick] (C) -- (B);
\draw[gray!50, thick] (C) -- (D);
\draw[gray!50, thick] (D) -- (E);
\draw[gray!50, thick] (A) -- (E);

\filldraw (A) circle (3pt);
\filldraw (B) circle (3pt);
\filldraw[gray!50] (C) circle (3pt);
\filldraw[gray!50] (D) circle (3pt);
\filldraw[gray!50] (E) circle (3pt);

\end{tikzpicture}
\end{center}
\caption{$I = \{i, i+1\}$}
\end{subfigure}
\begin{subfigure}{0.24\textwidth}
\begin{center}
\begin{tikzpicture}
\coordinate (A) at (-0.5,0);
\coordinate (B) at (0.5,0);
\coordinate (C) at (0.7,0.8);
\coordinate (D) at (0,1.2);
\coordinate (E) at (-0.7,0.8);

\draw[gray!50, thick] (A) -- (B);
\draw[gray!50, thick] (C) -- (B);
\draw[gray!50, thick] (C) -- (D);
\draw[gray!50, thick] (D) -- (E);
\draw[gray!50, thick] (A) -- (E);

\filldraw[gray!50] (A) circle (3pt);
\filldraw[gray!50] (B) circle (3pt);
\filldraw[gray!50] (C) circle (3pt);
\filldraw (D) circle (3pt);
\filldraw[gray!50] (E) circle (3pt);

\end{tikzpicture}
\end{center}
\caption{$\#I = 1$}
\end{subfigure}
\begin{subfigure}{0.24\textwidth}
\begin{center}
\begin{tikzpicture}
\coordinate (A) at (-0.5,0);
\coordinate (B) at (0.5,0);
\coordinate (C) at (0.7,0.8);
\coordinate (D) at (0,1.2);
\coordinate (E) at (-0.7,0.8);

\draw[gray!50, thick] (A) -- (B);
\draw[gray!50, thick] (C) -- (B);
\draw[gray!50, thick] (C) -- (D);
\draw[gray!50, thick] (D) -- (E);
\draw[gray!50, thick] (A) -- (E);

\filldraw[gray!50] (A) circle (3pt);
\filldraw[gray!50] (B) circle (3pt);
\filldraw[gray!50] (C) circle (3pt);
\filldraw[gray!50] (D) circle (3pt);
\filldraw[gray!50] (E) circle (3pt);

\end{tikzpicture}
\end{center}
\caption{$I = \emptyset$}
\end{subfigure}
\caption{Possible cases of $\K|_I$. Vertices not included in $I$ and edges not included in $\K|_I$ are painted in gray}
\end{figure}

Four of the examples shown in the middle of the figure have no nontrivial reduced cohomology, since they are homotopy equivalent to a point. Nontrivial cohomology arises in the case where $\K|_I = \K$, $\K|_I$ is two non-adjacent points, $\K|_I$ is two adjacent points plus one more not adjacent to either of them, and $\K|_I = \emptyset$. In the first case $\widetilde{H}^1(\K|_I) = \mathbb{C}$, in the second and third cases $\widetilde{H}^0(\K|_I) = \mathbb{C}$, and in the last case $\widetilde{H}^{-1}(\K|_I) = \mathbb{C}$.

Note that these are exactly the cases of $I$ in whose multigrading nontrivial values of $\Tor$ arose.

In addition, let us use this same example to illustrate Theorem \ref{main theorem structural sheaf}. Let us compute the cohomology of the canonical sheaf in degree zero, that is, of the sheaf $\Omega^5_{\UKC}$. Its cohomology in degree zero coincides with the cohomology of $\OO_{\UKC}$ in grading $m=(-1, \ldots, -1)$, that is, $I(m) = \{1, \ldots, 5\}$.

Let us show that this cohomology equals the reduced cohomology of $\K$ with a shift of indices, that is:
$$H^{-s, 10}(\UKC, \mathbb{C}) = \widetilde{H}^{5-s-1}(\K, \mathbb{C}) \cong H^{5-s}\lp\UKC, \Omega_{\UKC}^5\rp$$
That is, let us show that $\widetilde{H}^{i-1}(\K, \mathbb{C}) \cong H^{i}\lp\UKC, \Omega_{\UKC}^5\rp$

In this case, the \v{C}ech complex for $\Omega^5_{\UKC}$ can be built only from the affine charts corresponding to the maximal simplices: this gives 5 charts. The form $\omega = \lp\frac{\mathrm{d}x}{x}\rp_{\{1, \ldots, 5\}}$ is absent in all affine charts except the chart corresponding to the empty set (this chart is the open orbit of the torus action on the toric variety $\UKC$).

Let $\check{\mathcal{C}}^{\bullet}$ be the corresponding \v{C}ech complex. Then $\check{\mathcal{C}}^0 = 0$, since none of the charts corresponding to the simplices $\{i, i+1\}$ contain the form $\omega$. $\check{\mathcal{C}}^1 = \mathbb{C}^5$, with each summand $\mathbb{C}$ corresponding to a pair of charts $\{i,i+1\}, \{i+2, i+3\}$ that intersect only along the torus. For $i>1$, any $i+1$ charts corresponding to the simplex $\{i,i+1\}$ intersect only along the torus, so $\check{\mathcal{C}}^i = \mathbb{C}^{{5 \choose i+1}}$; in particular, starting from the term $\check{\mathcal{C}}^3$ the complex is exact, and hence acyclic. Its cohomology can only be nonzero in terms 1 or 2.

Note that $\chi\lp\check{\mathcal{C}}^{\bullet}\rp = 1-5+5=1$, since the complex $\mathbb{C}^{{5\choose i+1}}$ (starting from $i=-1$) is exact, and $\check{\mathcal{C}}^{\bullet}$ differs from it by 1 in term $-1$, by 5 in term 0, and by 5 in term 1. Note that $(1,5,5)$ is exactly the $f$-vector of $\K$. This is not a coincidence: it is a manifestation of the so-called combinatorial Alexander duality (see for reference \cite{Bjorner2009}). One could prove Theorem \ref{main theorem structural sheaf} by analyzing the Čech complexes for the structure sheaf and obtaining an isomorphism in cohomology with the reduced cohomology of $\K_I$ precisely via combinatorial Alexander duality, but such a proof would be noticeably more cumbersome than the one given in Section \ref{structural sheaf coh}.

From the value of the Euler characteristic it follows that $\dim H^2\lp\check{\mathcal{C}}^{\bullet}\rp = \dim H^1\lp\check{\mathcal{C}}^{\bullet}\rp + 1$. Let us show that the first differential is injective, that is, that $H^1\lp\check{\mathcal{C}}^{\bullet}\rp = 0$:

Note that among the triples $\{i,i+1\}, \{j,j+1\}, \{k,k+1\}$ 
there are triples of the form $\{i,i+1\}, \{i+2, i+3\}, \{i+1, i+2\}$. 
In such a triple, the only pair that appears in the previous term of the complex is the pair 
$\{i,i+1\}, \{i+2, i+3\}$. 
That is, the differential effects an embedding 
$$\check{\mathcal{C}}^{1} \xrightarrow{\mathrm{d}_1} \mathbb{C}\left\langle (\{i,i+1\}, \{i+2, i+3\}, \{i+1, i+2\})\;|\;i=1,\ldots, 5\right\rangle \subset \check{\mathcal{C}}^2$$
Therefore $\dd_1$ is injective.
This proves that
$$H^i(\check{\mathcal{C}}) = \begin{cases}
\mathbb{C},\text{ if } i=2,\\
0,\text{ otherwise}
\end{cases}$$
Recall that $\K_I = \K\cong S^1$, hence
$$\widetilde{H}^i(\K, \mathbb{C}) = \begin{cases}
\mathbb{C}, \text{ if } i=1,\\
0, \text{ otherwise}
\end{cases}
$$
This shows the desired coincidence of cohomology up to a shift:
$$\widetilde{H}^{i-1}(\K, \mathbb{C}) \cong H^{i}\lp\UKC, \Omega_{\UKC}^5\rp$$

\subsection{A moment-angle complex without Poincaré duality}\label{ex a4}
In the previous examples $\K$ was a triangulated sphere, in which case $\mathcal{Z}_{\K}$ is a manifold. In the general case of $\K$, $\mathcal{Z}_{\K}$ is only a $CW$-complex. The difference between a manifold and a $CW$-complex is clear: in the first case there is Poincaré duality, while in the second case it may be absent.

In the case of Example \ref{ex projective space}, $\mathcal{Z}_{\K}$ was a $(2n+1)$-dimensional manifold, in particular, $H^{k}(\mathcal{Z}_{\K}, \mathbb{C}) \cong$ $\cong H^{2n+1-k}(\mathcal{Z}_{\K}, \mathbb{C})$ (all zero except $H^0$ and $H^{2n+1}$).

In the case of Example \ref{ex pentagon} with the pentagon, $\mathcal{Z}_{\K}$ is a seven-dimensional manifold, so $$H^{7}(\UKC, \mathbb{C}) = H^{-3, 10}(\UKC, \mathbb{C}) \cong H^{0}(\UKC, \mathbb{C}) = H^{0,0}(\UKC, \mathbb{C})$$
$$H^{3}(\UKC, \mathbb{C}) = H^{-1, 4}(\UKC, \mathbb{C}) \cong H^{4}(\UKC, \mathbb{C}) = H^{-2, 6}(\UKC, \mathbb{C})$$

In this example let us consider $\K = A_4$, a one-dimensional complex which as a graph is isomorphic to the Dynkin graph $A_4$, that is, four vertices connected in sequence:

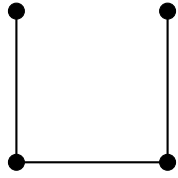
\begin{figure}[H]
\begin{center}
\begin{tikzpicture}
\coordinate (A) at (-1,2);
\coordinate (B) at (-1,0);
\coordinate (C) at (1,0);
\coordinate (D) at (1,2);

\filldraw (A) circle (3pt);
\filldraw (B) circle (3pt);
\filldraw (C) circle (3pt);
\filldraw (D) circle (3pt);

\draw[black, thick] (A) -- (B);
\draw[black, thick] (C) -- (B);
\draw[black, thick] (C) -- (D);
\end{tikzpicture}
\end{center}
\end{figure}

Its $f$-vector is $f(\K) = (1,4,3)$, so by Proposition \ref{prop number of points torus} the number of points in $U_{\K}$ equals:
$$\#\UKQ = (q-1)^4\cdot 1 + (q-1)^3\cdot 4 + (q-1)^2\cdot 3 = q^4 - 3q^2 + 2q$$

Already here we see that the symmetry present in the formula for the number of points in the previous examples is absent. Let us show that this asymmetry is also visible at the level of cohomology.
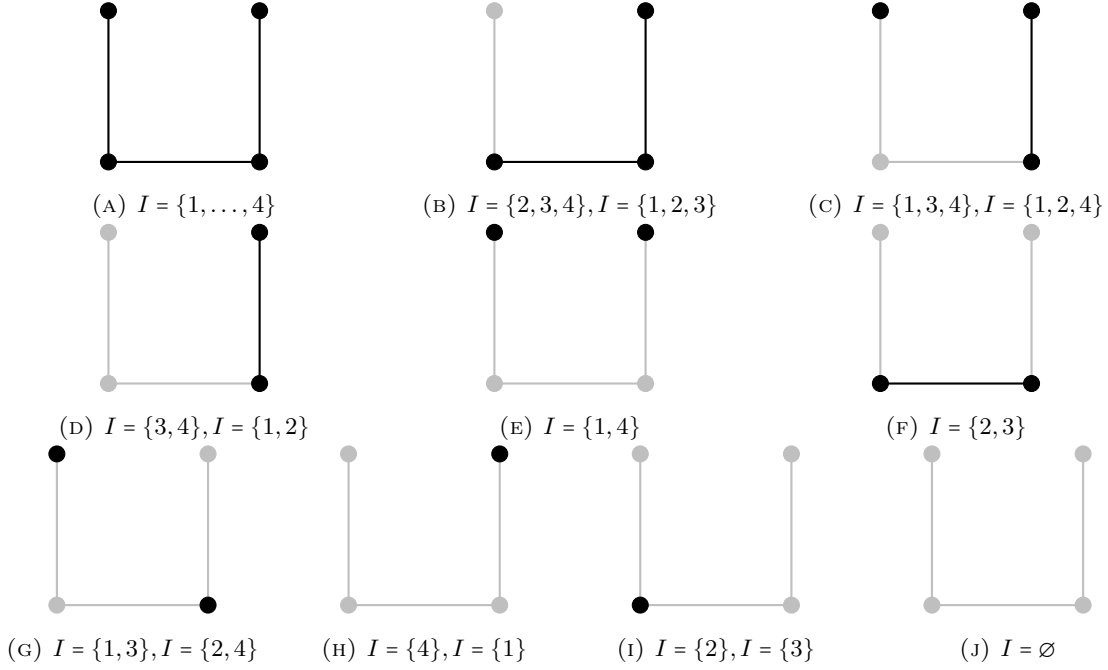
\begin{figure}
\begin{subfigure}{0.32\textwidth}
\begin{center}
\begin{tikzpicture}\coordinate (A) at (-1,2);
\coordinate (B) at (-1,0);
\coordinate (C) at (1,0);
\coordinate (D) at (1,2);

\draw[black, thick] (A) -- (B);
\draw[black, thick] (C) -- (B);
\draw[black, thick] (C) -- (D);

\filldraw (A) circle (3pt);
\filldraw (B) circle (3pt);
\filldraw (C) circle (3pt);
\filldraw (D) circle (3pt);
\end{tikzpicture}
\end{center}
\caption{$I = \{1, \ldots, 4\}$}
\end{subfigure}
\begin{subfigure}{0.32\textwidth}
\begin{center}
\begin{tikzpicture}\coordinate (A) at (-1,2);
\coordinate (B) at (-1,0);
\coordinate (C) at (1,0);
\coordinate (D) at (1,2);

\draw[black, thick][gray!50] (A) -- (B);
\draw[black, thick] (C) -- (B);
\draw[black, thick] (C) -- (D);

\filldraw[gray!50] (A) circle (3pt);
\filldraw (B) circle (3pt);
\filldraw (C) circle (3pt);
\filldraw (D) circle (3pt);
\end{tikzpicture}
\end{center}
\caption{$I = \{2, 3, 4\}, I = \{1,2,3\}$}
\end{subfigure}
\begin{subfigure}{0.32\textwidth}
\begin{center}
\begin{tikzpicture}\coordinate (A) at (-1,2);
\coordinate (B) at (-1,0);
\coordinate (C) at (1,0);
\coordinate (D) at (1,2);

\draw[gray!50, thick] (A) -- (B);
\draw[gray!50, thick] (C) -- (B);
\draw[black, thick] (C) -- (D);

\filldraw (A) circle (3pt);
\filldraw[gray!50] (B) circle (3pt);
\filldraw (C) circle (3pt);
\filldraw (D) circle (3pt);
\end{tikzpicture}
\end{center}
\caption{$I = \{1, 3, 4\}, I = \{1,2,4\}$}
\end{subfigure}

\begin{subfigure}{0.32\textwidth}
\begin{center}
\begin{tikzpicture}\coordinate (A) at (-1,2);
\coordinate (B) at (-1,0);
\coordinate (C) at (1,0);
\coordinate (D) at (1,2);

\draw[gray!50, thick] (A) -- (B);
\draw[gray!50, thick] (C) -- (B);
\draw[black, thick] (C) -- (D);

\filldraw[gray!50] (A) circle (3pt);
\filldraw[gray!50] (B) circle (3pt);
\filldraw (C) circle (3pt);
\filldraw (D) circle (3pt);
\end{tikzpicture}
\end{center}
\caption{$I = \{3, 4\},I=\{1,2\}$}
\end{subfigure}
\begin{subfigure}{0.32\textwidth}
\begin{center}
\begin{tikzpicture}\coordinate (A) at (-1,2);
\coordinate (B) at (-1,0);
\coordinate (C) at (1,0);
\coordinate (D) at (1,2);

\draw[gray!50, thick] (A) -- (B);
\draw[gray!50, thick] (C) -- (B);
\draw[gray!50, thick] (C) -- (D);

\filldraw (A) circle (3pt);
\filldraw[gray!50] (B) circle (3pt);
\filldraw[gray!50] (C) circle (3pt);
\filldraw (D) circle (3pt);
\end{tikzpicture}
\end{center}
\caption{$I = \{1,4\}$}
\end{subfigure}
\begin{subfigure}{0.32\textwidth}
\begin{center}
\begin{tikzpicture}\coordinate (A) at (-1,2);
\coordinate (B) at (-1,0);
\coordinate (C) at (1,0);
\coordinate (D) at (1,2);

\draw[gray!50, thick] (A) -- (B);
\draw[black, thick] (C) -- (B);
\draw[gray!50, thick] (C) -- (D);

\filldraw[gray!50] (A) circle (3pt);
\filldraw (B) circle (3pt);
\filldraw (C) circle (3pt);
\filldraw[gray!50] (D) circle (3pt);
\end{tikzpicture}
\end{center}
\caption{$I = \{2,3\}$}
\end{subfigure}

\begin{subfigure}{0.24\textwidth}
\begin{center}
\begin{tikzpicture}\coordinate (A) at (-1,2);
\coordinate (B) at (-1,0);
\coordinate (C) at (1,0);
\coordinate (D) at (1,2);

\draw[gray!50, thick] (A) -- (B);
\draw[gray!50, thick] (C) -- (B);
\draw[gray!50, thick] (C) -- (D);

\filldraw (A) circle (3pt);
\filldraw[gray!50] (B) circle (3pt);
\filldraw (C) circle (3pt);
\filldraw[gray!50] (D) circle (3pt);
\end{tikzpicture}
\end{center}
\caption{$I = \{1,3\}, I = \{2,4\}$}
\end{subfigure}
\begin{subfigure}{0.24\textwidth}
\begin{center}
\begin{tikzpicture}\coordinate (A) at (-1,2);
\coordinate (B) at (-1,0);
\coordinate (C) at (1,0);
\coordinate (D) at (1,2);

\draw[gray!50, thick] (A) -- (B);
\draw[gray!50, thick] (C) -- (B);
\draw[gray!50, thick] (C) -- (D);

\filldraw[gray!50] (A) circle (3pt);
\filldraw[gray!50] (B) circle (3pt);
\filldraw[gray!50] (C) circle (3pt);
\filldraw (D) circle (3pt);
\end{tikzpicture}
\end{center}
\caption{$I = \{4\}, I=\{1\}$}
\end{subfigure}
\begin{subfigure}{0.24\textwidth}
\begin{center}
\begin{tikzpicture}\coordinate (A) at (-1,2);
\coordinate (B) at (-1,0);
\coordinate (C) at (1,0);
\coordinate (D) at (1,2);

\draw[gray!50, thick] (A) -- (B);
\draw[gray!50, thick] (C) -- (B);
\draw[gray!50, thick] (C) -- (D);

\filldraw[gray!50] (A) circle (3pt);
\filldraw (B) circle (3pt);
\filldraw[gray!50] (C) circle (3pt);
\filldraw[gray!50] (D) circle (3pt);
\end{tikzpicture}
\end{center}
\caption{$I = \{2\}, I=\{3\}$}
\end{subfigure}
\begin{subfigure}{0.24\textwidth}
\begin{center}
\begin{tikzpicture}\coordinate (A) at (-1,2);
\coordinate (B) at (-1,0);
\coordinate (C) at (1,0);
\coordinate (D) at (1,2);

\filldraw[gray!50] (A) circle (3pt);
\filldraw[gray!50] (B) circle (3pt);
\filldraw[gray!50] (C) circle (3pt);
\filldraw[gray!50] (D) circle (3pt);

\draw[gray!50, thick] (A) -- (B);
\draw[gray!50, thick] (C) -- (B);
\draw[gray!50, thick] (C) -- (D);
\end{tikzpicture}
\end{center}
\caption{$I = \emptyset$}
\end{subfigure}
\caption{Possible cases of $\K|_I$. Vertices not included in $I$ and edges not included in $\K|_I$ are painted in gray}
\label{fig a4}
\end{figure}
As can be seen from Figure \ref{fig a4}, nontrivial reduced cohomology arises in the multigradings: $I = \{1,3,4\}, \{1,2,4\}, \{1,4\}, \{1,3\}, \{2,4\}$ give nontrivial $\widetilde{H}^0$, and $I = \emptyset$ gives nontrivial $\widetilde{H}^{-1}$. From Theorem \ref{weight theorem} we obtain (as in the previous example, let $h^{-s, 2t}$ denote $\dim H^{-s, 2t}(\UKC, \mathbb{C})$):
$$h^{-2, 6} = 2,\; h^{-1, 4} = 3,\; h^{0,0}=1$$

In particular, $\dim H^{4}(\mathcal{Z}_{\K}, \mathbb{C}) = 2,\;\dim H^{3}(\mathcal{Z}_{\K}, \mathbb{C}) = 3,\;\dim H^{0}(\mathcal{Z}_{\K}, \mathbb{C}) = 1$, and the remaining cohomology is zero,
that is, there is no Poincaré duality.

Since $\UKC$ is an eight-dimensional real manifold, passing to étale cohomology with compact support, we obtain:
$$h_{c,\text{\'et}}^{4,2} = 2,\;h_{c,\text{\'et}}^{5,4} = 3,\;h_{c,\text{\'et}}^{8,8} = 1$$
Applying the Grothendieck---Lefschetz formula:
$$\#\UKQ = (-1)^8\cdot q^{\frac{8}{2}}\cdot h_{c,\text{\'et}}^{8,8} + (-1)^5\cdot q^{\frac{4}{2}}\cdot h_{c,\text{\'et}}^{5,4} + (-1)^4\cdot q^{\frac{2}{2}}\cdot h_{c,\text{\'et}}^{4,2} = q^4 - 3q^2 + 2q$$

It is important to note that all the theorems in this work hold for the general case of $\K$, since nowhere do they use the assumption that $\K$ is a triangulation of a sphere, or that $\mathcal{Z}_{\K}$ is a manifold, or that $\K$ has no ghost vertices. Poincaré duality in étale cohomology applies to the algebraic variety $U_{\K}$, which is a smooth algebraic variety regardless of $\K$. Therefore, even for such exotic examples as this one, the formulas for the number of points in terms of the dimensions of the bigraded components, along with all the other results, remain valid. 	

\printbibliography

@book{buchstaber2014torictopology,
  title={Toric Topology},
  author={Buchstaber, V. and Panov, T.},
  isbn={9781470422141},
  lccn={15006771},
  series={Mathematical Surveys and Monographs},
  url={https://books.google.nl/books?id=CUo-CgAAQBAJ},
  year={2015},
  publisher={American Mathematical Society}
}

@ARTICLE{Danilov1978-bq,
  title     = "The geometry of toric varieties",
  author    = "Danilov, V.",
  journal   = "Russ. Math. Surv.",
  publisher = "IOP Publishing",
  volume    =  33,
  number    =  2,
  pages     = "97--154",
  month     =  apr,
  year      =  1978,
  url={https://www.mathnet.ru/php/archive.phtml?wshow=paper&jrnid=rm&paperid=3391}
}

@article{cattani2014hodge,
  title={Hodge Theory},
  author={Cattani, E. and El Zein, F. and Griffiths, P. and Trang, Le D.},
  year={2014},
  publisher={Princeton University Press},
  url={https://press.princeton.edu/books/ebook/9781400851478/hodge-theory-mn-49}
}

@article{Baskakov2002CohomologyOK,
  title={Cohomology of K-powers of spaces and the combinatorics of simplicial divisions},
  author={Baskakov, I.},
  journal={Russian Mathematical Surveys},
  year={2002},
  volume={57},
  pages={989 - 990},
  url={https://api.semanticscholar.org/CorpusID:250901936}
}

@Article{Weber2004,
author={Weber, A.},
title={Weights in the cohomology of toric varieties},
journal={Central European Journal of Mathematics},
year={2004},
month={Jun},
day={01},
volume={2},
number={3},
pages={478-492},
issn={1644-3616},
doi={10.2478/BF02475240},
url={https://doi.org/10.2478/BF02475240}
}

@Article{Bjorner2009,
author={Bj{\"o}rner, A.
and Tancer, M.},
title={Note: Combinatorial Alexander Duality---A Short and Elementary Proof},
journal={Discrete {\&} Computational Geometry},
year={2009},
month={Dec},
day={01},
volume={42},
number={4},
pages={586-593},
issn={1432-0444},
doi={10.1007/s00454-008-9102-x},
url={https://doi.org/10.1007/s00454-008-9102-x}
}

@book{cox2011toric,
  title={Toric Varieties},
  author={Cox, D. and Little, J. and Schenck, H.},
  volume={124},
  year={2011},
  publisher={American Mathematical Society},
  address={Providence, RI},
  series={Graduate Studies in Mathematics}
}

@article{Batyrev1991,
  title     = {On the classification of smooth complete toric varieties with Picard number 3},
  author    = {Batyrev, V.},
  journal   = {Mathematische Nachrichten},
  volume    = {161},
  pages     = {243--270},
  year      = {1993},
  publisher = {Wiley}
}

@misc{audin91,
	title={The Topology of Torus Actions on Symplectic Manifolds},
	author={Audin, M.},
	year={1991},
	publisher={Birkhäuser Basel},
	address={Birkhäuser Verlag Basel 1991},
	doi={https://doi.org/10.1007/978-3-0348-7221-8},
	url={https://link.springer.com/book/10.1007/978-3-0348-7221-8},
	issn={0743-1643},
	pages={184}
}

@article{Cox:1993fz,
    author = {Cox, D.},
    title = {The Homogeneous coordinate ring of a toric variety, revised version},
    eprint = {alg-geom/9210008},
    archivePrefix = {arXiv},
    month = {6},
    year = {1993}
}

@article{buchstaberpanov99,
author = {Buchstaber, V. and Panov, T.},
year = {1999},
month = {10},
pages = {},
title = {Torus actions and combinatorics of polytopes},
url={https://www.mathnet.ru/php/archive.phtml?wshow=paper&jrnid=tm&paperid=715}
}

@article{buchstaberpanov00,
  title     = "Moment-angle complexes and combinatorics of simplicial manifolds",
  author    = "Buchshtaber, V. and Panov, T.",
  journal   = "Russ. Math. Surv.",
  publisher = "Steklov Mathematical Institute",
  volume    =  55,
  number    =  3,
  pages     = "567--569",
  month     =  jun,
  year      =  2000,
  url={http://mi.mathnet.ru/rm297}
}

@article{eliyashev13,
author = {Eliyashev, Yu.},
year = {2013},
pages = {174–185},
title = {The Hodge filtration on complements of complex subspace arrangements and integral representations of holomorphic functions},
journal = {Журн. СФУ. Сер. Матем. и физ.},
url={https://m.mathnet.ru/php/archive.phtml?wshow=paper&jrnid=jsfu&paperid=312}
}

@article{eliyashev16,
  title     = "Mixed Hodge structure on complements of complex coordinate
               subspace arrangements",
  author    = "Eliyashev, Yu.",
  journal   = "Mosc. Math. J.",
  publisher = "National Research University, Higher School of Economics (HSE)",
  volume    =  16,
  number    =  3,
  pages     = "545--560",
  year      =  2016,
  url = {https://doi.org/10.17323/1609-4514-2016-16-3-545-560}
}

@misc{eliyashev13diss,
author = {Eliyashev, Yu.},
year = {2013},
title={Когомологии дополнений к наборам координатных подпространств и интегральные представления голоморфных функций},
url={https://research.sfu-kras.ru/node/8991}
}

@article{deligne1971,
author = {Deligne, P.},
year = {1971},
title = {Théorie de Hodge II},
journal = {Publications Math\'ematiques de l'IH\'ES},
pages={5–58},
volume={40},
url={https://publications.ias.edu/book/export/html/361}
}

@article{deligne1970,
author = {Deligne, P.},
year = {1970},
title = {Théorie de Hodge I},
journal = {Actes, Congr\'es intern. math.},
pages={425--430},
url={https://publications.ias.edu/sites/default/files/Theorie-de-Hodge-I.pdf}
}

@book{maninpanchishkin,
  title={Introduction to Modern Number Theory: Fundamental Problems, Ideas and Theories},
  author={Manin, Yu. and Panchishkin, A.},
  isbn={9783540276920},
  lccn={2007530690},
  series={Encyclopaedia of Mathematical Sciences},
  url={https://books.google.nl/books?id=wvK586IxaxwC},
  year={2006},
  publisher={Springer Berlin Heidelberg}
}

@book{barmak2011algebraic,
  title={Algebraic Topology of Finite Topological Spaces and Applications},
  author={Barmak, J.},
  lccn={2011934806},
  series={Lecture Notes in Mathematics},
  url={https://books.google.nl/books?id=clkLBbAjD2IC},
  year={2011},
  publisher={Springer Berlin Heidelberg}
}

@book{stanley1996combinatorics,
  title={Combinatorics and Commutative Algebra},
  author={Stanley, R.},
  number={v. 41},
  isbn={9780817643690},
  lccn={95025196},
  series={Comb. and Commutative Algebra},
  year={1996},
  publisher={Springer},
  url={https://link.springer.com/book/10.1007/b139094}
}

@misc{grothendieck66,
title={Formule de Lefschetz et rationalité des fonctions L},
journal={Séminaire Bourbaki},
author={Grothendieck, A.},
year={1964},
pages={41–55}
}

@book{milne1980,
  title={Étale Cohomology},
  author={Milne, J.},
  series={Princeton Mathematical Series},
  volume={33},
  year={1980},
  publisher={Princeton University Press},
  doi={10.1515/9781400883981}
}

\end{document}